\documentclass[11pt]{article}

\usepackage{amsmath,amsthm,amsfonts,amssymb,amsbsy}
\usepackage[numbers,sort&compress]{natbib}
\usepackage[dvipsnames]{xcolor}
\usepackage[colorlinks,citecolor=MidnightBlue,linkcolor=RedOrange,urlcolor=RedOrange]{hyperref}
\usepackage{graphicx}
\usepackage{bbm,comment,enumerate,longtable}
\usepackage{csquotes}
\usepackage{aliascnt}
\usepackage[noabbrev,capitalize]{cleveref}
\usepackage{authblk}
\usepackage[top=0.72in,bottom=0.72in,left=0.82in,right=0.82in]{geometry}

\usepackage{microtype}
\usepackage{titlesec}
\titlespacing*{\section}{0pt}{1.5ex plus 0.5ex minus 0.3ex}{0.8ex plus 0.2ex}
\titlespacing*{\subsection}{0pt}{1.2ex plus 0.4ex minus 0.2ex}{0.6ex plus 0.2ex}
\titlespacing*{\subsubsection}{0pt}{1.0ex plus 0.3ex minus 0.2ex}{0.5ex plus 0.1ex}

\makeatletter
\def\thm@space@setup{%
  \thm@preskip=5pt plus 2pt minus 2pt
  \thm@postskip=5pt plus 2pt minus 2pt
}
\makeatother

\AtBeginDocument{%
  \abovedisplayskip=7pt plus 2pt minus 3pt
  \belowdisplayskip=7pt plus 2pt minus 3pt
  \abovedisplayshortskip=3pt plus 2pt minus 1pt
  \belowdisplayshortskip=3pt plus 2pt minus 1pt
}

\usepackage{enumitem}
\setlist{topsep=2pt,itemsep=1pt,parsep=1pt}

\newcommand{\ntw}{\nu^{(2)}}
\newcommand{\E}{\mathbb{E}}
\newcommand{\tml}{\tilde{\mathfrak{L}}}

\newtheorem{thm}{Theorem}
\numberwithin{thm}{section}

\newaliascnt{lmm}{thm}
\newtheorem{lmm}[lmm]{Lemma}
\aliascntresetthe{lmm}
\crefname{lmm}{Lemma}{Lemmas}
\Crefname{lmm}{Lemma}{Lemmas}

\newaliascnt{cor}{thm}
\newtheorem{cor}[cor]{Corollary}
\aliascntresetthe{cor}
\crefname{cor}{Corollary}{Corollaries}
\Crefname{cor}{Corollary}{Corollaries}

\newaliascnt{prop}{thm}
\newtheorem{prop}[prop]{Proposition}
\aliascntresetthe{prop}
\crefname{prop}{Proposition}{Propositions}
\Crefname{prop}{Proposition}{Propositions}

\newaliascnt{problem}{thm}

\aliascntresetthe{problem}

\newtheorem{defn}{Definition}
\newtheorem{assm}{Assumption}
\theoremstyle{definition}
\newtheorem{remark}{Remark}
\newtheorem{ex}{Example}

\numberwithin{remark}{section}
\numberwithin{ex}{section}
\numberwithin{assm}{section}
\numberwithin{defn}{section}

\newcommand{\R}{\mathbb{R}}

\newcommand{\ee}{\mathbb{E}}

\newcommand{\N}{\mathbb{N}} 
\newcommand{\tb}{\tilde{b}}
\renewcommand{\P}{\mathbb{P}}

\newcommand{\tq}{\tilde{q}}
\newcommand{\cM}{\mathcal{M}}
\newcommand{\tp}{\tilde{p}}

\DeclareMathOperator*{\arginf}{arg\,inf}

\newcommand{\ta}{\tilde{a}}

\newcommand{\tH}{\widetilde{H}}

\numberwithin{equation}{section}

\newcommand{\mar}{\mathfrak{R}}
\newcommand{\tw}{\widetilde{W}}
\newcommand{\vep}{\varepsilon}

\allowdisplaybreaks
\title{LDP for Inhomogeneous U-Statistics}
\author[1]{Sohom Bhattacharya}
\author[2]{Nabarun Deb}
\author[3]{Sumit Mukherjee}
\affil[1]{Department of Statistics, University of Florida, \texttt{bhattacharya.s@ufl.edu}}
\affil[2]{Booth School of Business, University of Chicago, \texttt{nabarun.deb@chicagobooth.edu}}
\affil[3]{Department of Statistics, Columbia University, \texttt{sm3949@columbia.edu}}
\date{}

\begin{document}

\maketitle

\begin{abstract}
In this paper we derive a Large Deviation Principle (LDP) for  inhomogeneous U/V-statistics of a general order.
Using this, we derive a LDP for two types of statistics: random multilinear forms, and number of monochromatic copies of a subgraph. We show that the corresponding rate functions in these cases can be expressed as a variational problem over a suitable space of functions.  We use the tools developed to study Gibbs measures with the corresponding Hamiltonians, which include tensor generalizations of both Ising (with non-compact base measure) and Potts models. For these Gibbs measures, we establish scaling limits of log normalizing constants, and weak laws in terms of weak* topology, which are of possible independent interest.
\end{abstract}

\noindent\textit{MSC2020 subject classifications.} 60F10, 05C80, 82B20.

\section{Introduction}

\noindent Suppose ${\bf X}:=(X_1,\ldots,X_n)$ are i.i.d.~random variables from a non-degenerate probability measure $\mu$ on a Polish space $\mathcal{X}$ equipped with the Borel $\sigma$-field. Let $H=(V(H),E(H))$ be a finite graph with $v:=|V(H)|\ge 2$ vertices labeled $[v]=\{1,2,\ldots,v\}$, and maximum degree $\Delta$. Let $\phi:\mathcal{X}^v\mapsto \R$ be a measurable function, not necessarily symmetric in its arguments. In this paper we study the following inhomogeneous U-statistic :
\begin{align}\label{eq:U}
	U_n({\bf X}):=\frac{1}{n^v}\sum_{(i_1,\ldots,i_v)\in \mathcal{S}(n,v) }\phi(X_{i_1},\ldots,X_{i_v})\prod_{(a,b)\in E(H)}Q_n(i_a,i_b),
\end{align}
where $\mathcal{S}(n,v)$ is the set of all distinct tuples from $[n]^v$ (so that $|\mathcal{S}(n,v)|=v!\binom{n}{v}$), and $Q_n$ is a symmetric $n\times n$ matrix with $0$ on the diagonal. 
Related to $U_n({\bf X})$ is the inhomogeneous $V$-statistic:
\begin{equation}\label{eq:v}
V_n({\bf X}):=\frac{1}{n^v}\sum_{(i_1,\ldots ,i_v)\in [n]^v} \phi(X_{i_1},\ldots ,X_{i_v})\prod\limits_{(a,b)\in E(H)} Q_n(i_a,i_b).
\end{equation}
The only difference between $U_n({\bf X})$ and $V_n({\bf X})$ is that the sum is over a distinct set of indices $\mathcal{S}(n,v)$ in the former, as opposed to allowing them to repeat over the set $[n]^v$ in the latter. 
Several examples of statistics of interest can be expressed in the form \eqref{eq:U} and \eqref{eq:v}. Below we give some examples of inhomogeneous U-statistics, omitting the corresponding V-statistics version for brevity.

\begin{itemize}
	\item
	If $H=K_2$ is an edge and $\phi(x,y)=xy$, then 
	$$U_n({\bf X})=\frac{1}{n^2}\sum_{i\ne j} Q_n(i,j) X_i X_j$$ is a random quadratic form. In particular if $\mathcal{X}=\{-1,1\}$, then $U_n({\bf X})$ is the Hamiltonian/sufficient statistic for the Ising model~\cite{ising1925beitrag,Basak2017} (with coupling matrix $Q_n$), which is a Gibbs distribution on $\{-1,1\}^n$.
	
	\item
	If $H=K_3$ is a triangle, and $\phi(x,y,z)=xyz$, then
	$$U_n({\bf X})=\frac{1}{n^3}\sum_{i_1\ne i_2\ne i_3} Q_n(i_1,i_2) Q_n(i_2,i_3)Q_n(i_3,i_1)X_{i_1} X_{i_2}X_{i_3}$$ is a random cubic form and the Hamiltonian/sufficient statistic for the tensor Ising model (see~\cite{mukherjee2020estimation,mukherjee2020replica}; also see~\cref{sec:multlin}). 
	
	\item
	If $H=K_2$, and $\phi(x,y)=1\{x=y\}$, and $\mathcal{X}=\{1,2,\cdots,c\}$ is a finite set, then
	$$U_n({\bf X})=\frac{1}{n^2}\sum_{i\ne j} Q_n(i,j)1\{X_i=X_j\}$$
	is the Hamiltonian/sufficient statistic for the Potts model (which is a generalization of the Ising model) with $c$ colors (see~\cite{potts1952some,ellis1990limit}; also see~\cref{sec:tri}), and coupling matrix $Q_n$.
	
	\item
	If $H=K_3$, and $\phi(x,y,z)=1\{x=y=z\}$, then
	$$U_n({\bf X})=\frac{1}{n^3}\sum_{i_1\ne i_2\ne i_3} Q_n(i_1,i_2) Q_n(i_2,i_3) Q_n(i_3,i_1)1\{X_{i_1}=X_{i_2}=X_{i_3}\}.$$
	In this case, if $\mathcal{X}=\{1,2,\cdots,c\}$ is a finite set, and $Q_n$ is the (scaled) adjacency matrix of a graph $G_n$, then $U_n({\bf X})$ counts the (scaled) number of monochromatic triangles in $G_n$, where the vertices of $G_n$ receive colors from $\{1,2,\cdots,c\}$ according to the law $\mu$, independent of other vertices (see~\cref{sec:tri}). In particular, choosing $\mu$ as the uniform distribution on $\mathcal{X}=\{1,2,\cdots,c\}$ corresponds to a uniformly random coloring of the vertices.
\end{itemize}

\textcolor{black}{For the special choice $Q_n(i,j)=1\{i\ne j\}$, both expressions \eqref{eq:U} and \eqref{eq:v} reduce to the homogeneous versions of U/V-statistics, the \emph{Large Deviation Principle (LDP)}  for which is well known.(c.f.~\cite{arcones1992large,eichelsbacher1995large,eichelsbacher2002large,eichelsbacher2003moderate}). But otherwise, the LDP for a general random variable of the form \eqref{eq:U} or \eqref{eq:v} for a general sequence of matrices $\{Q_n\}_{n\ge 1}$ is not known, barring the very special case when $\mathcal{X}=\{0,1\}$, and $\{Q_n\}_{n\ge 1}$ is a sequence of \enquote{dense} matrices (i.e.~adjacency matrices of dense graphs), for the very special choice $\phi(x_1,\cdots,x_v)=\prod_{a=1}^v x_a$ (see~\cite[Thm 1.2]{mukherjee2020replica}). In particular, the LDP for the 4 examples outlined above is not known, despite these statistics receiving significant interest in the literature. The main result of this paper is a LDP for inhomogeneous U-statistics/V-statistics of this form with a good rate function in a very broad generality (see~\cref{thm:main} and~\cref{cor:main}), which allows $\mathcal{X}$ to be a general Polish space, $\phi(\cdot)$ to be an arbitrary function (satisfying) mild moment conditions (see~\eqref{eq:tailp}), and $\{Q_n\}_{n\ge1}$ to be a sequence of matrices which satisfy much milder restrictions (see~\eqref{eq:q} in~\cref{thm:main} and \eqref{eq:pp}/\eqref{eq:pp1}/\eqref{eq:pp2} in~\cref{cor:main}).}   
Below we recall the definition of an LDP (cf. \cite{DZ}):

\par 
Let $\{T_n\}_{n\ge 1}$ be a sequence of random variables taking values in a Polish space equipped with a Borel $\sigma$-field. We say that $\{T_n\}_{n\ge 1}$ satisfies a LDP {\color{black} with speed $n$ and }a good rate function $I(\cdot)$ taking values in $[0,\infty]$, if the level sets of $I(.)$ are compact, and for any Borel set $F$ we have
\begin{align*}
	-\inf_{x\in {F^\circ}}I(x)\le&\liminf_{n\to\infty}\frac{1}{n}\log \P(T_n\in F)\\
	\le &\limsup_{n\to\infty} \frac{1}{n}\log \P(T_n\in F)\le -\inf_{x\in \bar{F}}I(x),
\end{align*}
{\color{black} where $F^\circ$ and $\bar{F}$ denote the interior and closure of the set $F$ respectively (in $\R$).} This immediately allows us to study Gibbs measures with inhomogeneous U-statistics in the Hamiltonian (see ~\cref{lem:gibbs}).
\\

Utilizing our main result described above, we derive the LDP for two 
examples of interest, which are multilinear forms (see~\cref{thm:multilinear}) and number of monochromatic copies of subgraphs (see~\cref{thm:monochromatic}). In both these examples, we express the rate function as a constrained optimization problem over certain classes of functions. As applications, we study Gibbs distributions with multi-linear forms (see~\cref{thm:multilinear_gibbs}) and monochromatic copies of subgraphs (see~\cref{thm:monochromatic_gibbs}) as sufficient statistics.
As indicated above, this class of models generalizes the Ising and Potts models in several directions, going beyond quadratic interactions, compact state spaces, and uniform base measures. In both these cases, we compute the asymptotics of the log-partition function in terms of an optimization problem over suitable function classes. We also derive weak laws for these Gibbs measures, in the topology of weak* convergence .
\\

\subsection{Main results}\label{sec:mainres}
To establish the large deviation of $U_n({\bf X})/V_n({\bf X})$, we will assume throughout that the sequence of matrices $\{Q_n\}_{n\ge 1}$ converge in the weak cut distance. Cut distance/cut metric has been introduced in the combinatorics literature to study limits of graphs and matrices, see~\cite{FriezeKannan1999}, and have received significant attention in the recent literature (\cite{bc_lpi,borgs2018p,borgsdense1,borgsdense2}).  Below we formally introduce the notion of strong and weak cut distances used in this paper.

\begin{defn}\label{def:defirst}
	Suppose $\mathcal{W}$ is the space of all symmetric real-valued functions in $L^1([0,1]^2)$, symmetric in the arguments. Given two functions $W_1,W_2\in \mathcal{W}$, define the strong cut distance between $W_1, W_2$ by setting
	$$d_\square(W_1,W_2):=\sup_{S,T}\Big|\int_{S\times T} \Big[W_1(x,y)-W_2(x,y)\Big]dx dy\Big|.$$
	In the above display, the supremum is taken over all measurable subsets $S,T$ of $[0,1]$.  Define the weak cut distance by $$\delta_\square(W_1,W_2):= \inf_{\sigma} d_\square(W^\sigma_1,W_2)= \inf_{\sigma} d_\square (W_1,W^\sigma_2)$$ where $\sigma$ ranges from all measure preserving bijections $[0,1]\rightarrow [0,1]$ and $W^\sigma(x,y)=W(\sigma(x),\sigma(y))$.
	\\
	
	\noindent Given a symmetric matrix $Q_n$, define a function $W_{Q_n}\in \mathcal{W}$ by setting
	\begin{align*}
		W_{Q_n}(x,y)=& Q_n(i,j)\text{ if }\lceil nx\rceil =i, \lceil ny\rceil = j.
	\end{align*}
\end{defn}
{\color{black}For more details on cut metric and its manifold applications, we refer the interested reader to \cite{Lovasz2012}.} We now state two assumptions that we will use in many of our results.
\begin{assm}
	The sequence of matrices $\{Q_n\}_{n\ge 1}$ satisfy 
	\begin{align}\label{eq:cut_con}
		\delta_{\square}(W_{Q_n},W) \rightarrow 0,
	\end{align} for some $W\in \mathcal{W}$.
\end{assm}
\begin{assm}
	There exists a function $\psi:\mathcal{X}\mapsto [0,\infty)$ such that 
	\begin{align}\label{eq:tailp}
|\phi(x_1,\ldots,x_v)|\le \prod_{a=1}^v \psi(x_a),	\text{ and }	\E_{\mu} e^{\lambda \psi(X_1)^p}<\infty.
	\end{align}
	for all $\lambda>0$ and some $p\in [1,\infty]$ (where the case $p=\infty$ corresponds to assuming $\psi(\cdot)$ is uniformly bounded). 
\end{assm}
In particular if $\phi(x_1,\ldots,x_v)=\prod_{a=1}^v x_a$, then we can take $\psi(x)=|x|$, and if $\phi(x_1,\ldots,x_v)=1\{x_1=\ldots=x_v\}$, then we can take $\psi(.)\equiv 1$.
\\

To state our first main result, 
we require the following notations/definitions.

\begin{itemize}
	\item Let $\mathcal{M}$ denote the set of probability measures on $\Sigma:=[0,1]\times \mathcal{X}$, equipped with weak topology. Let $\widetilde{\mathcal{M}}$ denote the subset of measures in $\mathcal{M}$ where the first marginal is $\mathrm{U}[0,1]$. Note that $\widetilde{\mathcal{M}}$ is a closed subset of $\mathcal{M}$. Set 
	\begin{equation}\label{eq:mu0}\rho:=\mathrm{U}[0,1]\otimes \mu,\end{equation}
	and note that $\rho \in \widetilde{\mathcal{M}}$. {\color{black}Here $\mathrm{U}[0,1]$ denotes the Uniform distribution on $[0,1]$.}
	\item Let $D(\cdot|\cdot)$ be the Kullback-Leibler divergence between two probability measures on the same space. {\color{black} More precisely, for two measures $\kappa_1,\kappa_2$ on the same space such that $\kappa_1$ is absolutely continuous with respect to $\kappa_2$ with Radon-Nikodym derivative $\zeta$, set $D(\kappa_1|\kappa_2)= \E_{\kappa_1}(\log \zeta)$. If $\kappa_1$ is not absolutely continuous  with respect to $\kappa_2$, set  $D(\kappa_1|\kappa_2)= \infty$. }
	\item Given a topological space $\Gamma$, denote by $C_b(\Gamma)$ the space of all bounded, continuous functions on $\Gamma$.
	\item Given a probability measure $\nu\in\widetilde{\mathcal{M}}$, let $\nu^{(2)}$ denote its second marginal, i.e.,
	\begin{equation}\label{eq:secmar}
		\ntw(A)=\nu([0,1]\times A),
	\end{equation}
	for Borel subsets $A$ of $\mathcal{X}$. In particular, $\rho^{(2)}=\mu$.
\end{itemize}

{\color{black} Before stating the next definition, we recall that $H$ is a fixed graph with $v$ vertices and maximum degree $\Delta$.}

\begin{defn}
	For $W\in \mathcal{W}$, $\phi:\mathcal{X}^v\mapsto \mathbb{R}$, define a functional $T_{W,\phi}$ on $\widetilde{\mathcal{M}}$ by setting
	\begin{equation}\label{eq:definet}
		T_{W,\phi}(\nu):= \E\Bigg[\phi(B_1,\ldots,B_{v})\prod_{(a,b)\in E(H)}W(A_a,A_b)\Bigg],
	\end{equation}
	where the expectation is over $\{(A_a,B_a)\}_{1\le a\le v}\stackrel{i.i.d.}{\sim} \nu$. We note that since $\phi$ is possibly unbounded, this expectation may not exist in general for all $\nu$.
\end{defn}


\begin{thm}\label{thm:main}
	Suppose that  \eqref{eq:cut_con} and \eqref{eq:tailp} hold for some $p\in [1,\infty]$. Further, suppose there exists $q \in (1,\infty]$ such that 
	\begin{align}\label{eq:q}
		\limsup_{n\to\infty} \|W_{Q_n}\|_{q \Delta} < \infty,
	\end{align}
	and $\frac{1}{p}+\frac{1}{q}\le 1$. Then $U_n({\bf X})$ and $V_n({\bf X})$ (as in  \eqref{eq:U} and \eqref{eq:v} respectively) satisfy a LDP with the good rate function
	\begin{equation}\label{eq:i0}
		I_0(t)=\inf_{\nu\in \widetilde{\mathcal{M}}:\ D(\nu|\rho)<\infty,\ T_{W,\phi}(\nu)=t} D(\nu|\rho).
	\end{equation}	
\end{thm}
\noindent In~\cref{lem:basicres} part (ii) (in~\cref{sec:main}) we show that the function $T_{W,\phi}(.)$ is well-defined and finite-valued on the set $\{D(\cdot|\rho)<\infty\}$. Consequently, the infimum in~\eqref{eq:i0} is well-defined.

\begin{remark}\label{rem:implic}
Under assumptions~\eqref{eq:cut_con}~and~\eqref{eq:q}, ~\cite[Theorem 2.13]{borgs2018p} gives
\begin{equation}\label{eq:W_q}
    \|W\|_{q \Delta} < \infty,
\end{equation} 
for any $q>1$. This is something that we will use in the paper.
\end{remark}


\begin{remark}\label{rem:cramer} The case of homogeneous U-statistics corresponds to the choice $Q_n(i,j)=1_{i\ne j}$, for which \eqref{eq:q} holds with $q=\infty$. In this case~\eqref{eq:tailp} reduces to the assumption \begin{align}\label{eq:tailpp}
\E_\mu e^{\lambda \psi(X_1)}<\infty,
\end{align}
{\color{black} for all $\lambda>0$} which corresponds to the choice $p=1$. Assumptions similar to~\eqref{eq:tailpp} have been used in the literature for studying LDP for sample means (\cite[Thm 2.2.30]{DZ}), and (homogeneous) U-statistics (\cite[Eqn 1.4]{eichelsbacher2002large}), where it is referred to as the strong Cram\'{e}r condition.~\cref{thm:main} generalizes this to allow for inhomogeneity in the U-statistics, and allow for $q>1$ in~\eqref{eq:q}, as long as the corresponding $p$ in \eqref{eq:tailp} satisfies the natural Holder-type condition $\frac{1}{p}+\frac{1}{q}\le1$. 
\end{remark}

\noindent  For some special choices of $H$, the condition \eqref{eq:q} can be relaxed to allow for $q=1$, as shown in our second result. {\color{black} As explained in \cref{rem:sparse} below, this allows us to apply \cref{cor:main} to Erd\H{o}s-R\'{e}nyi graphs with $p_n \rightarrow 0$, whereas \cref{thm:main} is restricted to the case $p_n=p$ fixed (while allowing for an arbitrary subgraph $H$).}


\begin{thm}\label{cor:main}
	Suppose $\mathcal{X}$ is a compact metric space, and $\phi(\cdot)$ is continuous. Suppose further that \eqref{eq:cut_con} holds.  Then $U_n({\bf X})$ and $V_n({\bf X})$ (defined in \eqref{eq:U} and~\eqref{eq:v} respectively) satisfy a LDP with the same good rate function $I_0$ defined in \eqref{eq:i0}, if, setting
$$ \mar_{\tw}(x):=\int_{[0,1]} |\tw(x,y)|\,dy\text{ for }\tw\in \mathcal{W},$$ any of the following conditions hold {\color{black} with $U\sim\mathrm{U}[0,1]$}:
	\begin{itemize}
	    \item[(i)] If $H=K_2$ (edge) and 
	\begin{align}\label{eq:pp}
		\limsup_{n\to\infty} \E\left[\mar_{W_{Q_n}}(U)\right] < \infty,\quad  \E\left[\mar_W(U)\right] < \infty.
	\end{align}
	\item[(ii)] If $H=K_{1,v-1}$ (the $(v-1)$-star graph) and \begin{equation}\label{eq:pp1}\{\mar^{v-1}_{W_{Q_n}}(U)\}_{n\geq 1}\ \mbox{is uniformly integrable},\quad \E\left[\mar^{v-1}_{W}(U)\right]<\infty.\end{equation}
	\item[(iii)] If $H$ is a tree graph and 	
\begin{equation}\label{eq:pp2}\limsup_{n\to\infty}\lVert\mar_{W_{Q_n}}\rVert_{\infty}<\infty, \quad \lVert \mar_W\rVert_{\infty}<\infty.\end{equation}
		\end{itemize}
\end{thm}

\begin{remark}
In particular if $W_{Q_n}(.,.)\ge 0$, then we do not need to assume the second assumption on the limiting $W$ in each of the above displays \eqref{eq:pp}, \eqref{eq:pp1} and \eqref{eq:pp2}, as it follows from the corresponding assumption on $W_{Q_n}$. This is because cut metric convergence implies $\mar_{W_{Q_n}}(U)\stackrel{d}{\to}\mar_{W}(U)$. {\color{black} To see this, first replace $\delta_{\square}$ with $d_{\square}$ without loss of generality (by the same argument as used at the beginning of the proof of \cref{thm:main}). Consequently, $d_{\square}(W_{Q_n},W)\to 0$ and $W_{Q_n}(\cdot,\cdot)\ge 0$ implies  $$ \sup_{S\subseteq [0,1]} \bigg|\int_S (\mathfrak{R}_{W_{Q_n}}(u)-\mathfrak{R}_{W}(u))\,du\bigg|\to 0,$$ which in turn implies $\mar_{W_{Q_n}}(U)\stackrel{d}{\to}\mar_{W}(U)$.} It is unclear whether the above convergence holds without the non-negativity assumption.
\end{remark}
%
%
\noindent Note that in~\cref{cor:main} we did not need to assume \eqref{eq:tailp} separately, since $\mathcal{X}$ is compact and $\phi$ is continuous, implying boundedness of $\phi$. Part (i) of~\cref{cor:main} generalizes a result of \cite{BorgsLPII}, which specializes on the case $\mathcal{X}=[q]$ is a finite set (c.f.~\cite[Lem 3.2]{BorgsLPII} part (i) and \cite[Theorem 2.10]{BorgsLPII} part (v)). Our proofs are more direct {\color{black}(uses classical tools of LDP)}, and allow for a general choice of $H$ (compared to the case $H=K_2$ studied in \cite{BorgsLPII}, see parts (ii) and (iii) of~\cref{cor:main}, albeit under slightly stronger assumptions when $H\ne K_2$.

\noindent Given~\cref{thm:main} and~\cref{cor:main}, we will focus on $U_n({\bf X})$ throughout the remainder of the paper, noting that all our subsequent results apply to $V_n({\bf X})$ as well. {\color{black}Even though~\cref{thm:main} and~\cref{cor:main} both apply to general symmetric matrices, an important class of examples is when $Q_n$ is the scaled adjacency matrix of a simple labelled graph $G_n$, defined by $$Q_n:=\frac{1}{\|G_n\|_1}G_n, \text{ where }\|G_n\|_1:=\frac{1}{n^2}\sum_{i,j=1}^nG_n(i,j).$$
Here $G_n$ could be a convergent sequence of both deterministic and random graphs. For LDP of $U_n({\bf X})$ with a random $Q_n/G_n$, we will use the following definition.
\begin{defn}\label{def:ldprandomgraph}
    Let $Q_n$ be a random matrix.  
    We say that $U_n({\bf X})$ (defined in \eqref{eq:U}) satisfies a conditional LDP with speed $n$ and a \emph{deterministic} good rate function $I(\cdot)$, if the level sets of $I(\cdot)$ are compact, and for any Borel set $F$ and any $\vep>0$, we have: 
    $$\P_{Q_n}\left(-\inf_{x\in F^{\circ}} I(x)-\vep\le \frac{1}{n}\log\P(U_n({\bf X})\in F|Q_n)\le -\inf_{x\in \bar{F}} I(x)+\vep\right)\to 1.$$
\end{defn}
}

\begin{remark}\label{rem:sparse}
 \cref{thm:main} applies if $G_n$ is a convergent sequence of dense graphs (i.e.~$ \liminf_{n\to\infty}\|G_n\|_1>0)$. This is because \eqref{eq:q} holds with $q=\infty$, and so condition \eqref{eq:tailp} reduces to \eqref{eq:tailpp}. 
 {\color{black} In particular, \cref{thm:main} applies when $G_n$ is an Erd\H{o}s-R\'{e}nyi graph with parameter $p_n=p$ fixed, for an \textit{arbitrary} subgraph $H$ (including $H=K_3$) in the sense of \cref{def:ldprandomgraph}.} For a (non-graph) example where \eqref{eq:q} holds for some $q\in (1,\infty)$, we refer the reader to~\cref{rem:ex} below. On the other hand,~\cref{cor:main} allows for any sequence of graphs converging in cut metric (not necessarily dense), as in this case $\|W_{Q_n}\|_1=1$ by design. {\color{black}In particular,~\cref{cor:main} applies (in the sense of \cref{def:ldprandomgraph}) to an Erd\H{o}s-R\'{e}nyi graph $G_n$ on $n$ vertices with parameter $p_n$ as soon as $np_n\to \infty$ if $H=K_{1,v-1}$ is a star graph (including $H=K_2$), and whenever $np_n\gg \log{n}$ if $H$ is a non-star tree. This is because, with $(d_1(G_n),\cdots,d_n(G_n))$ denoting the labelled degree sequence, the above regimes of the parameter $p_n$ along with standard bounds using the Binomial distribution gives
 $$\frac{\sum_{i=1}^nd_i(G_n)}{n^2p_n}=O_P(1),\quad \frac{1}{n}\sum_{i=1}^n\frac{d_i^v(G_n)}{(np_n)^v}=O_P(1),\quad \max_{i\in [n]}\frac{d_i(G_n)}{np_n}=O_P(1). $$}
\end{remark}

 \begin{ex}\label{rem:ex}
 Fix $\alpha\in (0,1/\Delta)$, and let $$Q_n(i,j)= B_{ij}\left(\frac{ij}{n^2}\right)^{-\alpha},\quad W(x,y)= \frac12 (xy)^{-\alpha},\text{  where }B_{ij} \overset{\mathrm{iid}}{\sim} Bern\left(\frac12\right).$$
Then we have $$d_{\square}(W_{Q_n},W)=o_P(1), \text{ and }
\limsup_{n\to\infty} \|W_{Q_n}\|_{q\Delta} < \infty$$
for some $q>1$ (so the condition \eqref{eq:q} in~\cref{thm:main} does hold). However,  $$\|W_{Q_n}\|_\infty\stackrel{P}{\to} \infty,\text{ and }\|W_{Q_n}-W\|_1\stackrel{P}{\to}\frac{1}{2(1-\alpha)^2}>0.$$ Thus~\eqref{eq:q} does not hold with $q=\infty$, and one cannot replace convergence in cut metric in~\eqref{eq:q} by the stronger assumption of $L_1$ convergence. We defer the proof to~\cref{sec:auxlem}.\\
 \end{ex}

 \noindent As a first application of Theorem \ref{thm:main}, we study Gibbs distributions with inhomogeneous U-statistics as the Hamiltonian. Such models have been widely studied both in statistics and probability; see~\cref{sec:multlin} for relevant references. Before stating the result, we require the following definitions.
\begin{defn}\label{def:Gibbs}
	Fixing $\theta\in \R$, define a function $$Z_n(\theta):=\frac{1}{n}\log \E_{\mu^{\otimes n}} e^{n\theta U_n({\bf X})}\in (-\infty,\infty].$$ If $\theta$ is such that $Z_n(\theta)$ is finite, define a probability distribution $\R_{n,\theta}$ on $\mathcal{X}$ by setting
	\begin{align}\label{eq:gibbs}
		\frac{d\R_{n,\theta}}{d\mu^{\otimes n}}({\bf x})=\exp\Big(n\theta U_n({\bf x})-nZ_n(\theta)\Big).
	\end{align}
\end{defn}
\begin{defn}\label{def:lip}
	For two measures $\nu_1,\nu_2$ on $\Sigma$, define
	\[
	d_l(\nu_1,\nu_2):= \sup_{f \in \text{Lip}(1)}\Bigg|\int f d\nu_1-\int f d\nu_2\Bigg|,
	\]
	where the supremum is over the set of functions $f:\Sigma\mapsto [-1,1]$ which are $1$-Lipschitz. {\color{black} Equivalently, $d_l$ is the bounded Lipschitz metric on $\mathcal{M}$, the set of probability measures on $\Sigma$.}
\end{defn}

\noindent Armed with this, the following corollary of \cref{thm:main} computes the asymptotics of the log normalizing constant, and derives a weak law under $\R_{n,\theta}$ (defined in \eqref{eq:gibbs}) for the bivariate empirical measure $\mathfrak{L}_n$, defined by
\begin{equation}\label{eq:ln}
	\mathfrak{L}_n := \frac{1}{n}\sum_{i=1}^{n}\delta_{(\frac{i}{n},X_i)}\quad 
	\text{ for }{\bf X}=(X_1,\ldots ,X_n)\in \mathcal{X}^n.
\end{equation}
\begin{cor}\label{lem:gibbs}
	Suppose  \eqref{eq:cut_con} and \eqref{eq:q} hold for some $q\in (1,\infty]$, and \eqref{eq:tailp} holds for some $p\ge v$, and $\frac{1}{p}+\frac{1}{q}\leq 1$.
	%
	Then the following conclusions hold:
	\\ 
	(i)  For all  $\theta \in \mathbb{R}$ we have $$\lim_{n\to\infty}Z_n(\theta)=Z(\theta):=\sup_{\nu\in \widetilde{\mathcal{M}}:\ D(\nu|\rho)<\infty}\{\theta T_{W,\phi}(\nu)-D(\nu|\rho)\}<\infty.$$
	
	\noindent (ii) Under $\R_{n,\theta}$, ${\mathfrak{L}}_n$ satisfies a LDP with the good rate function
	\begin{equation}\label{eq:maingrf}
		I_\theta(\nu):=D(\nu|\rho)-\theta T_{W,\phi}(\nu)-\inf_{\sigma\in \widetilde{\mathcal{M}}:\ D(\sigma|\rho)<\infty}\{ D(\sigma|\rho)-\theta T_{W,\phi}(\sigma)\}.
	\end{equation} 
	Consequently, the supremum in part (i) is achieved on a set $F_\theta$, say, and further under $\R_{n,\theta}$, $$d_l(\mathfrak{L}_n,F_\theta)\stackrel{P}{\to}0.$$ 
\end{cor}

\begin{remark}
Note that~\cref{lem:gibbs} has the more restrictive condition $p\ge v$, instead of the condition $p\ge 1$ of~\cref{thm:main}. To see why this is necessary (and in fact the \enquote{right} condition), consider the case where $H=K_2$, $Q_n(i,j)=1_{i\ne j}$, $\phi(x,y)=xy$,   $\mu=N(0,1)$. This corresponds to studying the homogeneous U-statistic
$$U_n({\bf X})=\bar{X}^2-\frac{1}{n^2}\sum_{i=1}^nX_i^2\approx \bar{X}^2.$$
Since $\E e^{\lambda |X_1|}<\infty$ for all $\lambda>0$, one can study the LDP of $\bar{X}^2$ (via Cramer's theorem and contraction principle). On the other hand, one cannot define a Gibbs measure with $n\bar{X}^2$ as sufficient statistic, as $\E e^{n\theta \bar{X}^2}=\infty$ for all $\theta>1/2$.  Since $v=2$ for $H=K_2$,~\cref{lem:gibbs} instead demands that \eqref{eq:tailp} holds for some $p\ge 2$. Equivalently, we need the slightly stronger condition $\E e^{\lambda X_1^2}<\infty$ for all $\lambda>0$, which fails for $\mu=N(0,1)$.  

\end{remark}

\noindent The rate function in \eqref{eq:i0} in~\cref{thm:main} is obtained as a constrained optimization problem, whereas the rate function in part (ii) of~\cref{lem:gibbs} is in terms of an unconstrained optimization problem. The following Lemma expresses the optimizers of~\cref{thm:main} in terms of the optimizers obtained in~\cref{lem:gibbs}, which are typically easier to compute.

\begin{lmm}\label{lem:con_not}
Let the assumptions from~\cref{thm:main} hold. Suppose the set $T_{W,\phi}(F_\theta):=\{T_{W,\phi}(\nu),\nu\in F_\theta\}$ has cardinality $1$ for some $\theta$, and $Z(.)$ is differentiable at $\theta$.
\\

(i)  $Z'(\theta)= T_{W,\phi}(F_\theta).$
\\

(ii) Let $t$ be such that $Z'(\theta)=t$ for some $\theta$. Then $$\arginf_{\nu\in\widetilde{\mathcal{M}}:\ D(\nu|\rho)<\infty,\ T_{W,\phi}(\nu)=t} D(\nu|\rho)=F_\theta,$$
where $\rho$ is defined as in~\eqref{eq:mu0}.
\end{lmm}

\noindent We now apply our general results to study two classes of examples of interest.

\subsubsection{\textbf{Multilinear forms}}\label{sec:multlin}

Our first application focuses on multilinear forms, which correspond to the choice $\mathcal{X}=\R$  and $\phi(x_1,\ldots,x_v)=\prod_{a=1}^v x_a$ in \eqref{eq:U}. We will establish an LDP of the following quantity:
$$N_1(H,Q_n,{\bf X})=n^{-v}\sum_{(i_1,\ldots ,i_v)\in \mathcal{S}(n,v)} \left(\prod_{a=1}^v X_{i_a}\right)\left(\prod_{(a,b)\in E(H)}Q_n(i_a,i_b)\right)$$
In particular if $H=K_2, v=2$, then $N_1(H,Q_n,\bf{X})$ is a quadratic form. For stating the result, we require the following definitions.
\begin{defn}\label{def:tilt}
Let $\mu$ be a non-degenerate probability measure on $\R$, such that $$\alpha(\theta):=\log \int_{\mathcal{X}} e^{\theta x}d\mu(x)$$ is finite for all $\theta\in \R$. Define the $\theta$-exponential tilt of $\mu$ by setting $$\frac{ d\mu_{\theta}}{d \mu}(x):= \exp( \theta x - \alpha(\theta)).$$  Then the function $\alpha(.)$ is infinitely differentiable, with $$\alpha'(\theta)=\mathbb{E}_{\mu_\theta}(X),\quad \alpha''(\theta)=\mathrm{Var}_{\mu_\theta}(X)>0.$$  Consequently the function $\alpha'(.)$ is strictly increasing on $\R$, and has an inverse $\beta(.):\mathcal{N}\mapsto \R $, where $\mathcal{N}:=\alpha'(\R)$ is an open interval. Recall that $\bar{F}$ denotes the closure of a set $F$ in $\R$, and extend $\beta(\cdot)$ to a (possibly) infinite-valued function on $\bar{\mathcal{N}}$ by setting
\begin{align*}
	\beta(\sup\{\mathcal{N}\})=&+\infty\text{ if }\sup\{\mathcal{N}\}<\infty,\\
	\beta(\inf\{\mathcal{N}\})=&-\infty\text{ if }\inf\{\mathcal{N}\}>-\infty.
\end{align*}
Finally, define a function $\gamma:\beta(\bar{\mathcal{N}})\mapsto [0,\infty]$ by setting
\begin{eqnarray*}
	\gamma(\theta):=&D(\mu_{\theta}\| \mu)= \theta \alpha'(\theta)-\alpha(\theta)&\text{ if }\theta\in \R=\beta(\mathcal{N}),\\
	\gamma(\infty) :=&D(\delta_{\sup\{\mathcal{N}\}}|\mu)&\text{ if }\sup\{\mathcal{N}\}<\infty,\\
	\gamma(-\infty) :=&D(\delta_{\inf\{\mathcal{N}\}}|\mu)&\text{ if }\inf\{\mathcal{N}\}>-\infty.
\end{eqnarray*}

As an example, if $\mu=\frac{1}{2}(\delta_{-1}+\delta_1)$ is the symmetric Rademacher distribution, then we have 
$$\alpha(\theta)=\log \cosh(\theta),\quad \beta(u)=\tanh^{-1}(u),\quad \mathcal{N}=(-1,1).$$
\end{defn}

\begin{defn}\label{def:g1}
Fix $W\in \mathcal{W}$ such that $\|W\|_{q\Delta}<\infty$ for some $q>1$. Let $\mathcal{L}$ denote the space of all measurable functions $f:[0,1]\mapsto \bar{\mathcal{N}}$ such that $\int_0^1 |f(u)|^p\,du<\infty$, and $p\geq 1$, $q>1$ satisfies $\frac{1}{p}+\frac{1}{q}\leq 1$. Define the functional $G_{1,W}(.):\mathcal{L}\mapsto \R$ by setting
$$G_{1,W}(f):=\int_{[0,1]^v}\left(\prod_{(a,b)\in E(H)} W(x_a,x_b)\right) \left(\prod_{a=1}^v f(x_a)dx_a\right).$$
It follows from~\cref{lem:Tgraphon0} (by choosing $\phi(x_1,\ldots,x_v)=\prod_{a=1}^v f(x_a)$) that $G_{1,W}(f)$ is well-defined and finite for all $f\in\mathcal{L}$.
\end{defn}

\begin{defn}\label{def:xi1}
Define a map $\Xi_1:\mathcal{L}\mapsto \widetilde{\mathcal{M}}$ as follows:

For any $f\in \mathcal{L}$, if $(U,V)\sim \Xi_1(f)$, then $U\sim \mathrm{U}[0,1]$, and given $U=u$, one has
\begin{eqnarray*}
	V\sim &\mu_{\beta(f(u))}\text{ if }f(u)\in \mathcal{N},&\\
	=&\sup\{\alpha'(\R)\}\text{ if }f(u)=\sup\{\mathcal{N}\},&\text{ (this can only happen if }\sup\{\mathcal{N}\}<\infty),\\
	=&\inf\{\alpha'(\R)\}\text{ if }f(u)=\inf\{\mathcal{N}\},& \text{ (this can only happen if }\inf\{\mathcal{N}\}>-\infty).
\end{eqnarray*}
\end{defn}

We are now in a position to state the LDP for multilinear forms:

\begin{cor}\label{thm:multilinear}
Suppose that $(X_1,\cdots,X_n)\stackrel{i.i.d.}{\sim}\mu$ such that 
\begin{equation}\label{eq:pg1}
	\mathbb{E} e^{\lambda |X_1|^p}<\infty
\end{equation}
for some $p\in [1,\infty]$ and all $\lambda>0$,  and let $\{Q_n\}_{n\ge 1}$ be a sequence of matrices such that \eqref{eq:cut_con} and \eqref{eq:q} hold for some $W\in \mathcal{W}$ and $q>1$, $\frac{1}{p}+\frac{1}{q}\leq 1$. Then $G_{1,W}(f)$ is well-defined and finite for all $f\in\mathcal{L}$, and the random variable
$N_1(H,Q_n,{\bf X})$ satisfies a LDP with the good rate function 
\begin{equation}\label{eq:i1}
	I_1(t):=\inf_{f\in\mathcal{L}:\ \int_{[0,1]} \gamma(\beta(f(x)))dx<\infty,\ G_{1,W}(f)=t}  \int_{[0,1]} \gamma(\beta(f(x)))dx.
\end{equation}

\end{cor}

The upper tail LDP for $N_1(H,Q_n,G_n)$ was obtained in the special case when $\mu$ is a Bernoulli measure, and $Q_n$ is a dense graph, in \cite[Thm 1.2]{mukherjee2020replica}. ~\cref{thm:multilinear} generalizes this to arbitrary random variables, as well as for more general graphs/matrices.
\par

Utilizing the LDP in~\cref{thm:multilinear}, we study the behavior of Gibbs distributions with multilinear forms as Hamiltonians. Towards this direction, fix $\theta\in\R$ and let $\R_{n,\theta}^{(1)}$ be a probability distribution on $\R^n$ defined by setting
\begin{equation}\label{eq:highordu}\frac{d\R_{n,\theta}^{(1)}}{d\mu^{\otimes n}}({\bf x}):=\exp\Big(n\theta N_1(H,Q_n,{\bf x})-nZ_n^{(1)}(\theta)\Big),\end{equation}
where 
$$Z_n^{(1)}(\theta):=\frac{1}{n}\log \E_{\mu^{\otimes n}} e^{n\theta N_1(H,Q_n,{\bf X})}\in (-\infty,\infty].$$ 
In particular, this class of models generalizes the Ising model with quadratic interaction (i.e. $v=2$) on $\pm 1$ valued spins which has attracted a lot of attention in the Probability and Statistics literature (c.f.~\cite{Basak2017,AmirAndrea2010,Ellis1978,Sly2014,deb2020fluctuations,bhattacharya2025sharp} and references therein). The general $v$ spin model above for $v>2$ for $\pm 1$ valued spins has also gained significant attention in recent years (see~\cite{mukherjee2020estimation,barra2009notes,heringa1989phase,liu2019ising,suzuki1971zeros,turban2016one,yamashiro2019dynamics} and the references therein). In the following theorem, we allow for $v\ge 2$, and also let the spins take values in $\R$. In this setting, we study two fundamental features of model~\eqref{eq:highordu}, namely the scaling limit of the log normalizing constant (i.e., $Z_n^{(1)}(\cdot)$) and the weak limit of linear functionals of the empirical measure $\mathfrak{L}_n$.

\begin{thm}\label{thm:multilinear_gibbs}  Consider the same setting as in~\cref{thm:multilinear} with $p\geq v$. Then the following conclusions hold:

(i) $Z_n^{(1)}(\theta)$ is finite for all $n$ large enough, and 
$$\lim_{n\to\infty}Z_n^{(1)}(\theta)=\sup_{f\in\mathcal{L}:\ \int_{[0,1]} \gamma(\beta(f(x)))dx<\infty}\left\{\theta G_{1,W}(f)-\int_{[0,1]}\gamma(\beta(f(x)))dx\right\}.$$

(ii) The supremum in part (i) is achieved on a set $F^{(1)}_\theta\subset \mathcal{L}$ (say), which satisfies \begin{equation}\label{eq:multweaklim}
	d_l(\mathfrak{L}_n,\Xi_1(F^{(1)}_\theta))\overset{P}{\longrightarrow}0.
\end{equation}

(iii) With $p^*=\frac{p}{p-1}\le q$, for any $g(\cdot)\in L^{p^*}[0,1]$ the following holds: \begin{equation}\label{eq:weaaksta}\inf_{f_{\theta}\in F^{(1)}_{\theta}}\bigg|\int \omega_n(u)g(u)\,du - \int f_{\theta}(u)g(u)\,du\bigg|\overset{P}{\longrightarrow}0,\end{equation}
where $\omega_n:(0,1]\mapsto \R$ is a piecewise constant function defined by
\begin{equation}\label{eq:omegn}
	\omega_n(u)=X_i\text{ for }u\in \Big(\frac{i-1}{n},\frac{i}{n}\Big], \quad i\in [n].    
\end{equation}
\end{thm}

\begin{remark}[Connection to weak-* topology]\label{rem:weakstar}
If $F^{(1)}_{\theta}=\{f_{\theta}\}$ is a singleton, then~\cref{thm:multilinear_gibbs} part (iii) yields that the random function $\omega_n(\cdot)$ converges in weak-* topology on $L^p[0,1]$ to the non-random function $f_{\theta}(\cdot)$, in probability. As a consequence, we can get weak limits for a large class of linear combinations of the $X_i$'s. In particular, for any continuous function $g:[0,1]\mapsto \R$ we have 
$$\frac{1}{n}\sum_{i=1}^n g(i/n)X_i\overset{P}{\longrightarrow} \int_0^1 g(u) f_{\theta}(u)\,du.$$
\end{remark}

\noindent Note that part (i) of Theorem \ref{thm:multilinear_gibbs} expresses the asymptotics of the log normalizing constant as an optimization over functions, which is a simplification of the rate function obtained in Corollary \ref{lem:gibbs} part (i), which
expresses the rate function as an optimization problem over the space of measures.

\subsubsection{\textbf{Number of monochromatic copies of subgraphs}}\label{sec:tri}

As another application of Theorem \ref{thm:main}, we study the statistic
\begin{align}\label{eq:monochrome}
N_2(H,Q_n,\bf{X}):=&n^{-v}\sum_{(i_1,\ldots,i_v)\in \mathcal{S}(n,v)} 1\{ X_{i_1}=X_{i_2}=\ldots=X_{i_v}\} \prod_{(a,b)\in E(H)}Q_n(i_a,i_b),
\end{align}
where $(X_1,\ldots,X_n)\stackrel{i.i.d.}{\sim}\mu$, where $\mu$ is an arbitrary measure supported on $[c]=\{1,\ldots, c\}$, such that $\mu_r:=\P(X_1=r)>0$ for all $r\in [c]$. This corresponds to the choice $\mathcal{X}=[c]$, and $\phi(x_1,\ldots,x_v)=1\{x_1=\ldots=x_v\}$. In particular, if $Q_n$ is the scaled adjacency matrix of a graph $\mathcal{G}_n$, and $\mu$ is the uniform probability measure on $[c]$, then $N_2(H,Q_n,\bf{X})$ is the (scaled) number of monochromatic copies of the subgraph $H$ in $\mathcal{G}_n$, when $\mathcal{G}_n$ is colored uniformly at random with $c$ colors.
The asymptotic distribution of $N_2(H,Q_n,\bf{X})$ has been studied extensively in the literature for several choices of $Q_n$ (see \cite{Bhattacharya2017, dasgupta2005matching, diaconis2006methods, bhattacharya2022normal, arratia2016asymptotic, cerquetti2006poisson, bhattacharya2020second, bhattacharya2019monochromatic}). In particular, it is known that if $\mathcal{G}_n$ is a sequence of dense graphs (where the number of edges satisfy $|E(\mathcal{G}_n)|\propto n^2$) which converge in strong cut distance, then $N_2(H,Q_n,\bf{X})$ converges in distribution, after appropriate centering and scaling, to a random variable which is a (possibly infinite) weighted sum of independent centered $\chi^2_{c-1}$ random variables (see~\cite[Theorem 1.4]{Bhattacharya2017}~and~\cite[Theorem 1.3]{bhattacharya2019monochromatic}). The question of LDP for $N_2(H,Q_n,\bf{X})$ remained open, to the best of our knowledge. Our next result provides an answer to this question. Before stating it, we need the following definitions.
\begin{defn}\label{def:g2}
Let $\mathcal{F}_{c}$ denote the set of all measurable functions ${\bf f}=(f_1,\ldots,f_c):[0,1]\mapsto [0,1]^c$ such that $\sum_{r=1}^c f_r(x)=1$ for all $x\in [0,1]$.  For any $W\in \mathcal{W}$, define the functional $G_{2,W}(.):\mathcal{F}_c\mapsto \R$ by setting
$$G_{2,W}({\bf f}):=\sum_{r\in [c]}\int_{[0,1]^v}\prod_{(i,j)\in E(H)} W(x_i,x_j) \prod_{a=1}^v f_r(x_a)dx_a.$$
Note that $G_{2,W}({\bf f})<\infty$ for all ${\bf f}\in \mathcal{F}_c$ provided~\eqref{eq:q} holds. This follows from~\cref{lem:Tgraphon0} by choosing $\phi(x_1,\ldots,x_v)=\prod_{a=1}^v f(x_a).$
\end{defn}

\begin{defn}\label{def:xi2}
Define a map $\Xi_2:\mathcal{F}_c\mapsto \widetilde{\mathcal{M}}$, by the following construction: for any ${\bf f}\in \mathcal{F}_c$, we say $(U,V)\sim \Xi_2({\bf f})$ if $U\sim \mathrm{U}[0,1]$, and given $U=u$, let $V=r\in [c]$ with probability $f_r(u)$.

\end{defn}

We now state an LDP for $N_2(H,Q_n)$.

\begin{cor}\label{thm:monochromatic}
Suppose that \eqref{eq:cut_con} and \eqref{eq:q} hold for some $W\in \mathcal{W}$ and $q>1$. Then $N_2(H,Q_n)$ satisfies a LDP with the good rate function 
\begin{equation}\label{eq:i2}
I_2(t):=\inf_{{\bf f}\in \mathcal{F}_c:\ G_{2,W}({\bf f})=t}  \Big\{ \int_{[0,1]}\sum_{r=1}^c f_{r}(u)\log \frac{f_{r}(u)}{\mu_r}du\Big\}.
\end{equation}

\end{cor}

Utilizing the LDP in~\cref{thm:monochromatic}, we study the behavior of a Gibbs model with $N_2(H,Q_n)$ as sufficient statistic. Towards this direction, fix $\theta\in \R$, and let $\R_{n,\theta}^{(2)}$ be a probability distribution on $[c]^n$ defined by setting
\begin{equation}\label{eq:genpot}
\frac{d\R_{n,\theta}^{(2)}}{d\mu^{\otimes n}}({\bf x}):=\exp\Big(n\theta N_2(H,Q_n,{\bf x})-nZ_n^{(2)}(\theta)\Big),
\end{equation}
where 
$$Z_n^{(2)}(\theta):=\frac{1}{n}\log \E_{\mu^{\otimes n}} e^{n\theta N_2(H,Q_n,{\bf X})}\in (-\infty,\infty].$$
The case $v=2$ in model~\eqref{eq:genpot} above corresponds to the standard Potts model which has been studied extensively in the Probability and Statistics literature (c.f.~\cite{potts1952some,ellis1990limit,dembo2014replica,eichelsbacher2015rates}). In the following theorem, we study two fundamental features of model~\eqref{eq:genpot}, namely the scaling limit of the log normalizing constant (i.e., $Z_n^{(2)}(\cdot)$) and the weak limit of the empirical measure $\mathfrak{L}_n$.

\begin{thm}\label{thm:monochromatic_gibbs}
Consider the same setting as in~\cref{thm:monochromatic}.

(i) We have 
\begin{align*}\lim_{n\to\infty}Z_n^{(2)}(\theta)=&\sup_{\mathbf{f}\in \mathcal{F}_c}\left\{\theta G_{2,W}({\bf f})-\int_{[0,1]}\sum_{r=1}^c f_r(u)\log \frac{f_r(u)}{\mu_r}du\right\}
\end{align*}

(ii) The supremum in part (i) is achieved on a set $F^{(2)}_\theta\subset \mathcal{F}_c$, say. Then, under $\R_{n,\theta}^{(2)}$, we have:
$$d_l(\mathfrak{L}_n,\Xi_2(F^{(2)}_\theta))\overset{P}{\longrightarrow}0.$$ 

(iii) Setting $$\omega_{n,r}(x)=1\{X_i=r\}\text{ for }x\in \Big(\frac{i-1}{n},\frac{i}{n}\Big),\, r\in [c],$$
for any $\{g_r(\cdot)\}_{r\in [c]}$ with $g_r(\cdot)\in L^{1}[0,1]$, the following holds: $$\inf_{{\bf f}_{\theta}\in F^{(2)}_{\theta}}\max_{r\in [c]}\bigg|\int \omega_{n,r}(u)g_r(u)\,du - \int f_{r,\theta}(u)g_r(u)\,du\bigg|\overset{P}{\longrightarrow}0,$$


\end{thm}

\begin{remark}
   \cref{thm:monochromatic_gibbs} part (iii) proves a (vector) convergence of the function $\{\omega_{n,r}\}_{r\in [c]}$ to $\{\omega_{\infty,r}\}_{r\in [c]}$ in the special case $F^{(2)}_\theta=\{\omega_\infty\}$ is a singleton, similar to Theorem \ref{thm:multilinear_gibbs} part (iii). As a consequence, we get a weak limit for the empirical count vector
   $$\left(\frac{1}{n}\sum_{i=1}^n \mathbbm{1}(X_i=1),\ldots ,\frac{1}{n}\sum_{i=1}^n \mathbbm{1}(X_i=c)\right)\overset{P}{\longrightarrow}\left(\int_{0}^1 \omega_{\infty,1}(u)\,du,\ldots ,\int_{0}^1 \omega_{\infty,c}(u)\,du\right).$$
   However, note that the convergence in~\cref{thm:monochromatic_gibbs}, part (iii), is not exactly in weak$^*$ topology on $L_\infty^{[c]}$, as the dual of $L_\infty$ is strictly larger than $L_1$.
\end{remark}

In the special case when $Q_n$ is the adjacency matrix of a complete graph and $H=K_2$, the Gibbs model in \eqref{eq:genpot} reduces to the \emph{mean field} Potts model on the complete graph. The optimization problem of~\cref{thm:monochromatic_gibbs} part (i) has been studied extensively in this special case  (see~\cite{gandolfo2010limit,wu1982potts}). We plan to study the general optimization problem in a future paper.

\begin{remark}\label{rem:genen}
We have presented \cref{lem:gibbs}, \cref{lem:con_not}, \cref{thm:multilinear}, \cref{thm:monochromatic}, \cref{thm:multilinear_gibbs}, \cref{thm:monochromatic_gibbs}, in the setting of~\cref{thm:main}. All these results go through verbatim in the setting of~\cref{cor:main} without any change in the proof technique. We omit the details for brevity.
\end{remark}
\subsection{Overview of proof}

The starting point for the proofs is based on a simple but important LDP for the bivariate empirical measure $\mathfrak{L}_n=\frac{1}{n}\sum_{i=1}^n\delta_{i/n,X_i}$ with respect to weak topology, where $(X_1,\ldots,X_n)$ are IID from a measure $\mu$ on a Polish space $\mathcal{X}$ (c.f.~\cref{lmm:bivarldp}). Focusing on the V-statistic $V_n({\bf X})$, we can write it \enquote{approximately} as $T_{W_{Q_n},\phi}(\mathfrak{L}_n)$. Using the assumption \eqref{eq:tailp} along with Holder's inequality, we are able to replace $\phi$ by its truncation $\phi_M$ (\cref{lem:phim}). Under the assumption that $W_{Q_n}$ converges in weak cut metric to $W$, counting lemma for $L_q$ graphons (c.f.~\cref{lem:gen_holder}, \cite[Proposition 2.19]{borgs2018p}), we can replace $T_{W_{Q_n},\phi_M}(\mathfrak{L}_n)$ by $T_{W,\phi_M}(\mathfrak{L}_n)$ (\cref{lem:Tgraphon}). Finally, we show that $T_{W,\phi_M}(.)$ is a continuous function in weak topology (\cref{lem:Tcont}), and so an application of contraction principle gives the desired LDP for $V_n({\bf X})$. The proof of LDP for $U_n({\bf X})$ then proceeds via showing exponential equivalence (\cref{lem:phim}), thus completing the proof of~\cref{thm:main}.
The above proof requires that $\|W_{Q_n}\|_{q\Delta}$ is bounded (due to the dependence on the counting lemma in \cite[Proposition 2.19]{borgs2018p}). Focusing on the case when $H$ is a tree,~\cref{cor:main} derives the same LDP under the much weaker condition that $\|W_{Q_n}\|_1$ is bounded. In this case we use an alternative counting technique which bypasses the dependence on the maximum degree of $H$, but instead uses the tree structure. This allows~\cref{cor:main} to apply to sparse Erd\H{o}s-R\'enyi graphs, as soon as $np_n\to \infty$ (potentially up to logarithmic factors; see \cref{rem:sparse}).
\par

Focusing on the two applications of multilinear forms and monochromatic copies of subgraphs, using the separability of the function $\phi(.)$ in both cases, we show that the optimization problem in the rate function of the LDP (which in general is over the space of bivariate measures on $\Sigma=[0,1]\times \mathcal{X}$) is attained on the space of conditional exponential families, where the exponential families are specified in terms of a tilt with respect to the measure $\mu$. The rate function then simplifies as an optimization over the space of tilt functions, which is often more tractable. Using this machinery, we study Gibbs measures with these statistics as sufficient statistics, and derive asymptotics of the log normalizing constant (using Varadhan's Lemma). We also show that the optimizing tilt functions can be viewed as limits of a random function in weak* topologies. In particular, this can be used to establish weak laws for linear functions of ${\bf X}$ in multilinear form Gibbs model (c.f.~\cref{thm:multilinear_gibbs}), as well as the proportion vector in monochromatic copies Gibbs model (c.f.~\cref{thm:monochromatic_gibbs}). These Gibbs models generalize the celebrated Ising and Potts models in many ways, by allowing for interactions of a general order, non compact spaces, and non uniform coloring distributions.


\subsection{Future scope}
Our current LDP results require the assumption that $\|W_{Q_n}\|_{q\Delta}$ is bounded for some $q>1$, unless $H$ is a tree. In particular, this assumption rules out the case when $H=K_3$, and $Q_n$ is a Erd\H{o}s-R\'enyi graph with parameter $p_n$ (scaled by $p_n$), as soon as $p_n\to 0$. It remains to be seen whether this requirement is an artifact of our proof technique. A related problem is to study the LDP for both the statistics of interest (multilinear forms and monochromatic copies) in the case when $Q_n$ converges in a mode different from cut metric (such as graphs converging in local weak topology).
Focusing on optimization, it is of interest to analyze the resulting constrained optimization problem in the LDP for both examples, and derive the conditional behavior of the system under rare events. A related problem is to study the optimization problem arising in Gibbs measures. In both these cases, an interesting question is to study the regimes of symmetry/symmetry breaking (i.e.~the optimizing functions are constants/not constants respectively). In particular, constant optimizers correspond to  product probability measures on $\Sigma=[0,1]\times \mu$, thereby revealing additional structure (universal behaviors similar to \cite[Thm 2.1]{Basak2017}). The nuances with optimization has been partially studied by the authors in a subsequent draft in more detail ~\cite{bhattacharya2023gibbs}.

{\color{black}\subsection{Acknowledgment} We thank Bhaswar Bhattacharya for helpful discussions at various stages of  the manuscript. We also thank the Associate Editor and the two anonymous referees for their comments and suggestions, which greatly improved the presentation of the paper.}

\section{Proofs of Main results}\label{sec:main}

\noindent A first step of our proofs is the following extension of Sanov's theorem (c.f.~\cite[Theorem 6.2.10]{DZ}). 
For stating the lemma, we define the following (good) rate function 
\begin{equation}\label{eq:grfi}
I(\nu)= \begin{cases}
D(\nu|  \rho), \qquad \text{if } \nu\in \widetilde{\mathcal{M}},\\
\infty \qquad \text{otherwise}
\end{cases}
\end{equation}
\begin{lmm}\label{lmm:bivarldp}
\begin{enumerate}
\item[(i)] With {\color{black}$(\mathfrak{L}_n)_{n\ge 1}$} as in~\eqref{eq:ln}, $\mathfrak{L}_n$ satisfies a LDP on $\mathcal{M}$ with respect to weak topology, with the good rate function $I$ defined by \eqref{eq:grfi}.

\item[(ii)]
Define a sequence of bivariate measures {\color{black}$(\tml_n)_{n\ge 1}\in \widetilde{\mathcal{M}}$} as the joint law of $(U,  X_{\lceil nU\rceil})$ conditional on ${\bf X}$, where $U\sim \mathrm{U}[0,1]$. Then $\tilde{\mathfrak{L}}_n$ satisfies a LDP on $\widetilde{\mathcal{M}}$ with respect to weak topology, with the good rate function $D(\cdot| \rho)$.
\end{enumerate}
\end{lmm}

\begin{proof} 
\begin{enumerate}
\item[(i)]
To show $\mathfrak{L}_n$ satisfies a LDP, we will use  \cite[Cor 4.6.14]{DZ}. To this effect, we claim that the limit
\begin{equation*}
\mathcal{T}(f):=\lim\limits_{n \rightarrow \infty} \frac{1}{n} \log \ee\Big(e^{\sum_{i=1}^nf(\frac{i}{n},X_i)}\Big)
\end{equation*}
{\color{black} exists as a finite real number for all $f \in C_b(\Sigma)$ (recall that $C_b(\Sigma)$ is the space of all bounded continuous functions from $\Sigma=[0,1]\times\mathcal{X}$ to $\R$), and the functional $\mathcal{T}(.)$ is Gateaux differentiable}. To this end, note that,
{\color{black}\begin{align*}
\frac{1}{n} \log \ee(e^{\sum_{i=1}^{n}f(\frac{i}{n},X_i)})&=\frac{1}{n}\sum_{i=1}^n \log\int_{\mathcal{X}} e^{f(\frac{i}{n},x)}\,d\mu(x)\nonumber \\ &
\rightarrow \int_{0}^{1} \left(\log \int_{\mathcal{X}} e^{f(u,x)}d\mu(x)\right)du.
\end{align*}
as $f$ is bounded and continuous on $\Sigma=[0,1]\times\mathcal{X}$.}  
Therefore, {\color{black}$$\mathcal{T}(f)=\int_{0}^{1} \left(\log \int_{\mathcal{X}} e^{f(u,x)}d\mu(x)\right)du$$}
is well-defined and finite for all $f \in C_b(\Sigma)$. Moreover, for $f,g \in C_b(\Sigma)$, {\color{black} and note that the function $$\Lambda_{f,g}(u,t) := \log \int_{\mathcal{X}} e^{f(u,x)+tg(u,x)}d\mu(x)$$ is differentiable at $t=0$ with 
$$\Lambda'_{f,g}(u,0)= \Big[\frac{\int_{\mathcal{X}} g(u,x)e^{f(u,x)} d\mu(x)}{\int_{\mathcal{X}} e^{f(u,x)} d\mu(x)}\Big]$$
invoking \cite[(2.2.9)]{DZ}. Consequently, using the fact that $g$ is bounded, it follows by DCT that $t \mapsto \mathcal{T}(f+tg)= \int_{0}^{1} \Lambda_{f,g}(u,t) du$ is also differentiable at $t=0$ and 
 \[
 \frac{d}{dt}\mathcal{T}(f+tg)|_{t=0}= \int_{0}^{1} \Big[\frac{\int_{\mathcal{X}} g(u,x)e^{f(u,x)} d\mu(x)}{\int_{\mathcal{X}} e^{f(u,x)} d\mu(x)}\Big] du,
 \]}
Thus $\mathcal{T}(.)$ is Gateaux differentiable. Also, $\mathfrak{L}_n$ is exponentially tight, as both the marginals $\frac{1}{n}\sum_{i=1}^n\delta_{\frac{i}{n}}$ and $\frac{1}{n}\sum_{i=1}^n\delta_{X_i}$ are exponentially tight (see~\cite[Lemma 6.2.6]{DZ}). Recalling that $\Sigma=[0,1]\times \mathcal{X}$ and invoking \cite[Cor 4.6.14]{DZ}, $\mathfrak{L}_n$ satisfies LDP with rate function given by 
\begin{equation}\label{r1}
\mathcal{T}^\star(\nu)= \sup_{f \in C_b(\Sigma)}\Big\{\int_{\Sigma} f(u,x)d\nu(u,x) - \mathcal{T}(f)\Big\}, \qquad \nu \in \mathcal{M}.
\end{equation}
Finally, we need to check that $\mathcal{T}^\star \equiv I$. Since $I(.)$ is convex, using the duality of Legendre transform, with $\rho$ as in \eqref{eq:mu0}, it is enough to show:
\begin{align}\label{eq:suffice1}
\mathcal{T}(f)= \sup_{\nu\in \mathcal{M}}\Big\{\int_{\Sigma} f(u,x)d\nu(u,x) - I(\nu)\Big\}=\sup_{\nu\in \widetilde{\mathcal{M}}}\Big\{\int_\Sigma fd\nu- D(\nu|\rho)\Big\},
\end{align}
Choosing $\nu$ such that $d\nu(x|u)= \frac{e^{f(u,x)}}{C(u)}$ where $C(u):= \int_{\mathcal{X}}  e^{f(u,x)} d \mu(x)$, {\color{black}observe that $\frac{d\nu}{d\rho}(u,x)=\frac{e^{f(u,x)}}{C(u)}$ and so 
\begin{align*}
    \int f\,d\nu&=\int_{\Sigma} \frac{f(u,x)e^{f(u,x)}}{C(u)}\,d\rho(u,x),\\
    D(\nu|\rho)&=\int_{\Sigma} \frac{e^{f(u,x)}}{C(u)} (f(u,x) - \log C(u)) d\rho(u,x).
\end{align*}
Using the definitions of $C(\cdot)$ and $\mathcal{T}(\cdot)$, it follows that 
\begin{align*}
\int_{\Sigma} fd\nu- D(\nu|\rho)= &\int_{\Sigma} \frac{e^{f(u,x)} \log C(u)}{C(u)}\,d\rho(u,x)\\
=&\int_{0}^{1} \log C(u)du= \int_{0}^{1} \left(\log \int_{\mathcal{X}} e^{f(u,x)}d\mu(x)\right)du=\mathcal{T}(f).
\end{align*}}
Thus, we have shown LHS $\le$ RHS, in \eqref{eq:suffice1}. Towards the other direction, {\color{black}for any $\nu$ absolutely continuous with respect to $\rho$}, setting $\phi(u,x):=\frac{d\nu}{d\rho}(u,x)$ note that $\phi(u,.)d\mu$ is a probability measure on $\mathcal{X}$ for any $u\in [0,1]$, as the first marginal of $\nu$ is $\mathrm{U}[0,1]$. Thus, an application of Jensen's inequality gives
\begin{align*}
\log \int_{\mathcal{X}} e^{f(u,x)} d\mu(x)&=  \log \int_{\mathcal{X}}\left(\frac{ e^{f(u,x)}}{\phi(u,x)}\right)\phi(u,x) d\mu(x)\\ &\ge \int_{\mathcal{X}} (f(u,x)-\log \phi(u,x)) \phi d\mu(x).
\end{align*}
On integrating this over $u$ gives
\begin{align*}\mathcal{T}(f) &\geq \int_{\Sigma}  f(u,x) \phi(u,x) d\rho(u,x)-\int_{\Sigma} \phi(u,x) \log \phi(u,x) d\rho(u,x)\\ &=\int_{\Sigma} f(u,x)d\nu(u,x) - D(\nu|\rho).
\end{align*}
The bound LHS $\ge$ RHS in \eqref{eq:suffice1} follows from this, on taking a sup over {\color{black}$\nu$ absolutely continuous with respect to $\rho$. This completes the proof of part (i), on noting that $D(\nu|\rho)=\infty$ if $\nu$ is not absolutely continuous with respect to $\rho$.}

\item[(ii)]

With $\tml_n$ as defined in the theorem, we first show that $\mathfrak{L}_n$ and $\tml_n$ are exponentially equivalent. Indeed, with $d_l(\cdot,\cdot)$ denoting the bounded Lipschitz metric as in Definition \ref{def:lip}, we have
\begin{align*}
d_l(\mathfrak{L}_n,\tml_n) &= \sup_{f \in \text{Lip}(1)}\left|\frac{1}{n}\sum_{i=1}^{n}f\Big(\frac{i}{n},X_i\Big)-\sum_{i=1}^{n}\int_{\frac{i-1}{n}}^{\frac{i}{n}} f(u,X_i)du\right| \\
&\le  \sup_{f \in \text{Lip}(1)}\sum_{i=1}^{n}\int_{\frac{i-1}{n}}^{\frac{i}{n}} \Big|f\Big(\frac{i}{n},X_i\Big)-f(u,X_i)\Big|du \leq \frac{1}{n}.
\end{align*}
Hence, invoking \cite[Theorem 4.2.13]{DZ}, $\tml_n$ and $\mathfrak{L}_n$ have the same LDP. Finally note that $I(\nu)=\infty$ outside $\widetilde{\mathcal{M}}$, and $\widetilde{\mathcal{M}}$ is closed.  The desired conclusion then follows on invoking part (i). 

\end{enumerate}
\end{proof}
%
%
%
%
%
%

\subsection{Proofs of Theorem \ref{thm:main},~\cref{cor:main},~\cref{lem:gibbs},~and~\cref{lem:con_not}}


For proving~\cref{thm:main}~and~\cref{cor:main} we require the following four lemmas, the proofs of which are deferred to~\cref{sec:auxlem}. 

The first lemma gives a bound on $T_{W,\phi}(.)$. This will be used repeatedly throughout the paper.
\begin{lmm}\label{lem:Tgraphon0}

For any $p\ge 1$, $q>1$ such that $\frac{1}{p}+\frac{1}{q}\leq 1$, we have
\begin{align*}
T_{|W|,|\phi|}(\nu)\le \|W\|_{q \Delta}^{|E(H)|} \Big(\E_{(\ntw)^{\otimes v}} |\phi(B_1,\ldots,B_v)|^p\Big)^{\frac{1}{p}}.
\end{align*}
where $\{B_r\}_{1\le r\le v}\stackrel{iid}{\sim}\nu^{(2)}$, and $\nu^{(2)}$ is the second marginal of $\nu$ (see~\eqref{eq:secmar}).
\end{lmm}


\begin{defn}\label{def:phim}
For any measurable $\phi:\mathcal{X}^ v\mapsto \R$ and $M>0$, define $\phi_M:=\phi 1\{|\phi|\le M\}$. Let $U_{n,M}({\bf X})$ and $V_{n,M}({\bf X})$ denote the RHS of \eqref{eq:U} and \eqref{eq:v} respectively, where the function $\phi$ is replaced by $\phi_M$. Note that $V_{n,M}({\bf X})=T_{W_{Q_n},\phi_M}(\tilde{\mathfrak{L}}_n)$.
\end{defn}
\noindent {\color{black} Setting
	\begin{equation}\label{eq:ak}
		\mathcal{A}_K:=\{\nu\in \widetilde{M}:D(\nu|\rho)\le K\},
	\end{equation}
 the second lemma allows us to replace $\phi$ by the bounded function $\phi_M$.}
\begin{lmm}\label{lem:phim}

Suppose $\phi:\mathcal{X}^v\mapsto \R$ is a measurable function, such that the pair $(\phi,\mu)$ satisfy \eqref{eq:tailp}. Then we have the following conclusions:

(i) For every $\delta>0$ we have
$$\limsup\limits_{M \rightarrow \infty} \limsup\limits_{n \rightarrow \infty} \frac{1}{n}\log \mathbb{P}\left(|V_n({\bf X})-V_{n,M}({\bf X})| \geq \delta\right) = - \infty,$$
$$\limsup\limits_{M \rightarrow \infty} \limsup\limits_{n \rightarrow \infty} \frac{1}{n}\log \mathbb{P}\left(|U_n({\bf X})-U_{n,M}({\bf X})| \geq \delta\right) = - \infty.$$

(ii) For every $K>0$ we have
$$\lim_{M\to\infty}\sup_{\nu\in\mathcal{A}_K} |T_{W,\phi}(\nu)-T_{W,\phi_M}(\nu)|=0,$$
where $\mathcal{A}_K$ is defined as in~\eqref{eq:ak}.

\end{lmm}

\noindent The third and fourth lemma analyze the functional $T_{W,\phi}(.)$ for bounded $\phi$.
\begin{lmm}\label{lem:Tcont}
Let $W \in \mathcal{W}$ and $\phi:\mathcal{X}^v\mapsto[-M,M]$ be a measurable bounded function.

(i) Define 
$$
W^*(u_1,\ldots ,u_v):=\prod_{(a,b)\in E(H)} W(u_a,u_b),\qquad \mbox{for} \, (u_1,\ldots ,u_v)\in [0,1]^v.
$$
Suppose $\phi$ is continuous and $\lVert W^*\rVert_1<\infty$. Then the function $\nu \mapsto T_{W,\phi}(\nu)$ is continuous on  $ \widetilde{\mathcal{M}}$.
\\

(ii)  Suppose $\phi$ is measurable, $W\in\mathcal{W}$ satisfies $\|W\|_{q \Delta} < \infty$ for some $q>1$,  
and $\{\phi^{(\ell)}\}_{\ell\ge 1}:\mathcal{X}^v\mapsto [-M,M]$ be such that $\phi^{(\ell)} \stackrel{L^p(\mu^{\otimes v})}{\longrightarrow}\phi.$ 
Then we have
$$\lim_{\ell\to\infty}\sup_{\nu \in \mathcal{A}_K}|T_{W,\phi^{(\ell)}}(\nu)-T_{W,\phi}(\nu)|=0,$$
where $\mathcal{A}_K$ is defined as in~\eqref{eq:ak}.

\end{lmm}

\begin{lmm}\label{lem:Tgraphon}

Suppose $\{Q_n\}_{n\ge 1}$ be a {\color{black}sequence of matrices such that  $d_{\square}(W_{Q_n},W) \rightarrow 0$ for some $W\in \mathcal{W}$}, and \eqref{eq:q} holds for some $q>1$.
Then, for any bounded, continuous function $\phi$ and $\delta>0$, we have \begin{equation}
\lim_{n \rightarrow \infty} \frac{1}{n}\log \P \left(|T_{W_{Q_n},\phi}(\tilde{\mathfrak{L}}_n)-T_{W,\phi}(\tilde{\mathfrak{L}}_n)| \geq \delta \right)= -\infty.
\end{equation}
\end{lmm}

The following lemma allows us to get moment bounds under $\ntw$ using moment bounds under $\mu$. 
\begin{lmm}\label{lem:basicres}
	Assume that~\eqref{eq:tailp} holds. Then we have the following:
	\begin{enumerate}
		\item[(i)]
		For every positive number $K$ and function $\mathfrak{G}:(0,\infty)\mapsto(0,\infty)$, there exists a finite positive number $L$ (depending only on $K,\mathfrak{G}$) with the following property:
			
			For any measurable $\eta:\mathcal{X}\mapsto \mathbb{R}$ such that  $\mathbb{E}_{\mu}(e^{\lambda|\eta|^p})\le \mathfrak{G}(\lambda)$ for all $\lambda >0$, 
			we have 
			$$\sup_{\nu \in \mathcal{A}_K} \mathbb{E}_{\ntw}(|\eta|^p)\le L \mathbb{E}_{\rho}(|\eta|^p)$$
			
		\item[(ii)] Assume $\lVert  W\rVert_{q \Delta}<\infty$ (where $\Delta$ is the maximum degree of the graph $H$ as before) for some $q>1$ such that $\frac{1}{p}+\frac{1}{q}\leq 1$. Then $T_{W,\phi}(\cdot)$ is well-defined on the set $D(\cdot|\rho)<\infty$. Further, we have:
		\begin{equation}\label{eq:welldef}\sup_{\nu\in\mathcal{A}_K} |T_{W,\phi}(\nu)|\leq C_2<\infty,\end{equation}
where $C_2>0$ be a constant depending on $K,\mu,\psi, \|W\|_{q \Delta}$.
	\end{enumerate}
\end{lmm}

		The next lemma will be used to show that $U_n({\bf X})$ and $V_n({\bf X})$ have the same LDP.
		
		\begin{lmm}\label{lmm:un_vn_same}

Suppose $Q_n$ satisfies~\eqref{eq:q} for some $q>1$. Let $\phi:\mathcal{X}^v\to [-L,L]$ for some $L>0$. Then,
\begin{align*}
   \lim\limits_{n \rightarrow \infty} \frac{1}{n^v}\sup\limits_{(x_1,\ldots ,x_n) \in \mathcal{X}^n}& \bigg\lvert \sum_{(i_1,\ldots,i_v)\in \mathcal{S}(n,v) } \left(\prod_{(a,b)\in E(H)}Q_n(i_a,i_b)\right)\phi(x_{i_1},\ldots ,x_{i_v})-\\ &\sum_{(i_1,\ldots,i_v)\in [n]^v}\left(\prod_{(a,b)\in E(H)}Q_n(i_a,i_b)\right)\phi(x_{i_1},\ldots ,x_{i_v})\bigg\rvert= 0,
\end{align*}

\end{lmm}

\begin{proof}[Proof of Theorem \ref{thm:main}]
By definition of the weak-cut convergence (see~\cref{def:defirst}), there exists a sequence of permutations $\{\pi_n\}_{n\ge 1}$ with $\pi_n\in S_n$ such that $$d_\square(W_{Q_n^{\pi_n}}, W) \rightarrow 0,\text{ where }Q_n^{\pi_n}(i,j):=Q_n(\pi_n(i),\pi_n(j)).$$
Then setting $Y_i=X_{\pi_n(i)}$ we have
\begin{align*}
    U_n({\bf X})=&\frac{1}{n^v}\sum_{(i_1,\ldots,i_v)\in \mathcal{S}(n,v)}\left(\prod_{(a,b)\in E(H)}Q_n(i_a,i_b)\right)\phi(X_{i_1},\ldots,X_{i_v})\\
    =&\frac{1}{n^v}\sum_{(i_1,\ldots,i_v)\in \mathcal{S}(n,v)}\left(\prod_{(a,b)\in E(H)}Q_n(\pi_n(i_a),\pi_n(i_b))\right)\phi\Big(X_{\pi_n(i_1)},\ldots,X_{\pi_n(i_v)}\Big)\\
    =&\frac{1}{n^v}\sum_{(i_1,\ldots,i_v)\in \mathcal{S}(n,v)}\left(\prod_{(a,b)\in E(H)}Q_n(\pi_n(i_a),\pi_n(i_b))\right)\phi\Big(Y_{i_1},\ldots,Y_{i_v}\Big)\\
    \stackrel{D}{=}&\frac{1}{n^v}\sum_{(i_1,\ldots,i_v)\in \mathcal{S}(n,v)}\left(\prod_{(a,b)\in E(H)}Q_n(\pi_n(i_a),\pi_n(i_b))\right)\phi\Big(X_{i_1},\ldots,X_{i_v}\Big),
\end{align*}
where the last equality uses the fact that $(Y_1,\ldots,Y_n)\stackrel{D}{=}(X_1,\ldots,X_n).$ Since $d_\square(W_{Q_n^{\pi_n}},W)\to 0$, by replacing $W_{Q_n}$ by $W_{Q_n^{\pi_n}}$ without loss of generality we can assume $d_\square(W_{Q_n},W)\to 0$, which we do throughout the rest of this proof.

Next, we show the exponential equivalence between $U_n(\bf X)$ and $V_n({\bf X})$. Towards this direction, fix $M>0$, and let $\phi_M(\cdot), U_{n,M}({\bf X}), V_{n,M}({\bf X})$ be as in~\cref{def:phim}. Then it suffices to show the following, for any $\delta>0$:
\begin{equation}\label{eq:expeq1}
\lim\limits_{M\to\infty}\limsup\limits_{n\to\infty}\frac{1}{n}\log{\P\left(|V_n({\bf X})-V_{n,M}({\bf X})|\geq\delta\right)}=-\infty,
\end{equation}
\begin{equation}\label{eq:expeq2}
\lim\limits_{M\to\infty}\limsup\limits_{n\to\infty}\frac{1}{n}\log{\P\left(|U_n({\bf X})-U_{n,M}({\bf X})|\geq\delta\right)}=-\infty,
\end{equation}
\begin{equation}\label{eq:expeq3}
\lim\limits_{n\to\infty}\frac{1}{n}\log{\P\left(|U_{n,M}({\bf X})-V_{n,M}({\bf X})|\geq\delta\right)}=-\infty.
\end{equation} Here~\eqref{eq:expeq1}~and~\eqref{eq:expeq2} are direct consequences of~\cref{lem:phim} part (i). Finally~\eqref{eq:expeq3} follows directly from~\cref{lmm:un_vn_same} part (i). This establishes exponential equivalence between $U_n({\bf X})$ and $V_n({\bf X})$, and so by~\cite[Theorem 4.2.13]{DZ} it suffices to derive the LDP of $V_n({\bf X})=T_{W_{Q_n},\phi}(\tilde{\mathfrak{L}}_n)$. In the remainder of the proof, we will use the notation $T_{W_{Q_n},\phi}(\tilde{\mathfrak{L}}_n)$ and $T_{W_{Q_n},\phi_M}(\tilde{\mathfrak{L}}_n)$ instead of $V_n({\bf X})$ and $V_{n,M}({\bf X})$ respectively.
With $\phi_M$ as in Definition \ref{def:phim}, use standard measure theory arguments to get a sequence of continuous functions $\{\phi^{(\ell)}_{M}\}_{\ell\ge 1}$ such that $$|\phi^{(\ell)}_{M}|\le M,\text{ and } \phi^{(\ell)}_{M}\stackrel{L^p(\mu^{\otimes v})}{\longrightarrow}\phi_M.$$ 
Then we claim that
\begin{equation}\label{eq:phi_equivalence2}
\limsup\limits_{\ell \rightarrow \infty} \limsup\limits_{n \rightarrow \infty} \frac{1}{n}\log \mathbb{P}\left(|T_{W_{Q_n},{\phi}_M}(\tilde{\mathfrak{L}}_n)-T_{W_{Q_n},\phi^{(\ell)}_{M}}(\tilde{\mathfrak{L}}_n)| \geq \delta\right) = - \infty.
\end{equation}
We first complete the proof of the theorem, deferring the proof of 
\eqref{eq:phi_equivalence2}.
{\color{black} To this effect, triangle inequality gives 
{\small \begin{align*}
    \Big|T_{W,\phi^{(\ell)}_{M}}(\tilde{\mathfrak{L}}_n)- T_{W_{Q_n},\phi}(\tilde{\mathfrak{L}}_n)\Big| &\le |T_{W,\phi^{(\ell)}_{M}}(\tilde{\mathfrak{L}}_n)-T_{W_{Q_n},\phi^{(\ell)}_{M}}(\tilde{\mathfrak{L}}_n)|\\
    &+  |T_{W_{Q_n},{\phi}_M}(\tilde{\mathfrak{L}}_n)-T_{W_{Q_n},\phi^{(\ell)}_{M}}(\tilde{\mathfrak{L}}_n)|+ \Big|T_{W_{Q_n},\phi}(\tilde{\mathfrak{L}}_n)-T_{W_{Q_n},{\phi}_M}(\tilde{\mathfrak{L}}_n)\Big| 
    \\
    &=: \tau^{(1)}_{n,\ell,M}+\tau^{(2)}_{n,\ell,M}+ \tau^{(3)}_{n,M}
\end{align*}}
Consequently, for any $\delta >0$, union bound gives
\begin{align*}
   \P (\Big|T_{W,\phi^{(\ell)}_{M}}(\tilde{\mathfrak{L}}_n)- T_{W_{Q_n},\phi}(\tilde{\mathfrak{L}}_n)\Big|> 3\delta) \le 3 \max \{ \P(\tau^{(1)}_{n,\ell,M} >\delta),  \P(\tau^{(2)}_{n,\ell,M}>\delta),  \P(\tau^{(3)}_{n,M}>\delta)\}.
\end{align*}
Using Lemma \ref{lem:Tgraphon},  \eqref{eq:phi_equivalence2} and Lemma \ref{lem:phim} part (i), respectively, the above display gives, for every $\delta>0$,
$$\limsup_{M\to\infty}\limsup_{\ell\to\infty}\limsup_{n\to\infty}\frac{1}{n}\log \P\left(\Big|T_{W_{Q_n},\phi}(\tilde{\mathfrak{L}}_n)-T_{W,\phi^{(\ell)}_{M}}(\tilde{\mathfrak{L}}_n)\Big|>3\delta\right)=-\infty,$$}
i.e.~the random variables $T_{W,\phi^{(\ell)}_{M}}(\tilde{\mathfrak{L}}_n)$ are exponentially good approximations of $T_{W_{Q_n},\phi}(\tilde{\mathfrak{L}}_n)$ (in the iterated limit as $\ell\to\infty$ followed by $M\to\infty$). Next, invoking~\cref{lem:Tgraphon0}, with $\phi\equiv 1$, we get: \begin{equation}\label{eq:l1w}
\lVert W^*\rVert_1\leq \lVert W\rVert_{q\Delta}^{|E(H)|}<\infty,
\end{equation}
where the last inequality uses~\eqref{eq:W_q}. Consequently~\cref{lem:Tcont}, part (i) implies that the map $T_{W,\phi^{(\ell)}_{M}}(.)$ is continuous with respect to weak topology on $\widetilde{\mathcal M}$.
Also invoking~\cref{lem:phim} part (ii) and~\cref{lem:Tcont} part (ii), we have
\begin{align}\label{eq:later}\limsup_{M\to\infty}\limsup_{\ell\to\infty}\sup_{\nu\in \mathcal{A}_K}|T_{W,\phi^{(\ell)}_{M}}(\nu)-T_{W,\phi}(\nu)|=0,
\end{align}
and so \cite[Eq (4.2.24)]{DZ} holds. Since by Lemma \ref{lmm:bivarldp}, $\tilde{\mathfrak{L}}_n$ satisfies a LDP with the good rate function $D(\cdot|\rho)$ on $\widetilde{M}$, it follows on invoking \cite[Theorem 4.2.23]{DZ} that $T_{W,\phi}(\tilde{\mathfrak{L}}_n)$ satisfies a LDP with good rate function $I_0$. For invoking \cite[Theorem 4.2.23]{DZ} we need the function $\nu\mapsto T_{W,\phi}(\nu)$ to be well defined and finite on the set $\{\nu\in \widetilde{\cM}:D(\nu|\rho)<\infty\}$. But this follows from~\cref{lem:basicres} part (ii).
\\

To complete the proof, it remains to show  \eqref{eq:phi_equivalence2}.

			\noindent To this effect, setting $\widetilde{\phi}_{\ell,M}:=\phi_M-\phi^{(\ell)}_{M}$ and using Lemma \ref{lem:Tgraphon0} gives
			\begin{align*}
				|T_{W_{Q_n},\phi_{M}}(\tilde{\mathfrak{L}}_n)-T_{W_{Q_n},\phi^{(\ell)}_{M}}(\tilde{\mathfrak{L}}_n)|&=|T_{W_{Q_n},\widetilde{\phi}_{\ell,M}}(\widetilde{\mathcal{L}}_n)|\\
				&\le \|W_{Q_n}\|_{q\Delta}^{|E(H)|} \Big(\E_{\tilde{\mathfrak{L}}_n} |\widetilde{\phi}_{\ell,M}(B_1,\ldots,B_v)|^p\Big)^{\frac{1}{p}},
			\end{align*}
			where the first term is bounded in $n$, by \eqref{eq:q}. For showing \eqref{eq:phi_equivalence2} it thus suffices to show that for every $\delta>0$ we have
			\begin{align}\label{eq:phi6}
				\limsup_{\ell\to\infty}\limsup_{n\to\infty}\frac{1}{n}\log \P(|\E_{\tilde{\mathfrak{L}}_n} |\widetilde{\phi}_{\ell,M}(B_1,\ldots,B_v)|^p|>\delta)=-\infty.
			\end{align}
			To this effect, we begin with the following claim, whose proof we defer.
			\\

			There exists a sequence of positive reals $\{\sigma_\ell\}_{\ell\ge 1}$ converging to $0$, such that 
			\begin{equation}\label{eq:claimsubG}
				\limsup_{\ell\to\infty} \E\left[\exp\left(\frac{\widetilde{\phi}_{\ell,M}^2(X_1,\ldots ,X_v)}{8\sigma_{\ell}^2}\right)\right]\leq 2.
			\end{equation}
			In other words,~\eqref{eq:claimsubG} shows that the sub-Gaussian norm of $\phi_{\ell,M}(X_1,\ldots ,X_v)$ converges to $0$ as $\ell\to\infty$. Using~\cite[Lemma 3.1]{eichelsbacher1995large}, we get the bound
			$$\P\bigg(\big|\E_{\tilde{\mathfrak{L}}_n} |\widetilde{\phi}_{\ell,M}(B_1,\ldots,B_v)|^p-\E_{\mu^{\otimes v}}\big|\widetilde{\phi}_{\ell,M}^p(X_1,\ldots ,X_v)\big|\big|>\delta\bigg)\leq \exp\left(-\frac{n\delta^2}{2v\sigma_{\ell}^2}\right).$$
			Further, an application of the dominated convergence theorem yields $$\lim_{\ell\to\infty}\E_{\mu^{\otimes v}}\big|\widetilde{\phi}_{\ell,M}^p(X_1,\ldots ,X_v)\big| = 0.$$ Combining this observation with the above display,  completes the proof of~\eqref{eq:phi6}~and consequently also verifies~\eqref{eq:phi_equivalence2}.
			
			
			It thus remains to verify \eqref{eq:claimsubG}. To this effect, for any $\sigma>0$, the random variable $\widetilde{\phi}_{\ell,M}^2(X_1,\ldots ,X_v)/\sigma^2\overset{\mathbb{P}}{\longrightarrow}0$ as $\ell\to\infty$. By the dominated convergence theorem, we then have $$\lim_{\ell\to\infty}\E_{\mu^{\otimes v}}\left[\exp\left(\frac{\widetilde{\phi}_{\ell,M}^2(X_1,\ldots ,X_v)}{8\sigma^2}\right)\right]=1.$$ Since this holds for every $\sigma>0$, it follows that there exists a sequence $\{\sigma_\ell\}_{\ell\ge 1}$ of positive reals converging to $0$, such that $$\lim_{\ell\to\infty}\E_{\mu^{\otimes v}}\left[\exp\left(\frac{\widetilde{\phi}_{\ell,M}^2(X_1,\ldots ,X_v)}{8\sigma_{\ell}^2}\right)\right]=1.$$
			This verifies \eqref{eq:claimsubG}, and hence completes the proof of the theorem.
			\\

		\end{proof}
		
		\noindent  The following lemma gives the requisite modification of Lemma \ref{lem:Tcont} part (ii) and~\cref{lem:Tgraphon} for proving~\cref{cor:main}.

\begin{lmm}\label{lem:quadratic}
Suppose $\mathcal{X}$ is a compact metric space,  $\phi(\cdot)$ is continuous, and $d_{\square}(W_{Q_n},W)\to 0$. Suppose further that either  the assumptions in~\cref{cor:main}, part (i), or part (ii), or part (iii) holds. Then $T_{W,\phi}(\cdot)$ is well-defined and finite, and we have the following conclusions:

(i) $\sup_{\nu\in \widetilde{\mathcal{M}}}|T_{W_{Q_n},\phi}(\nu)-T_{W,\phi}(\nu)|\to 0.$
\\

(ii) The map $\nu\mapsto T_{W,\phi}(\nu)$ is continuous in $\widetilde{\mathcal{M}}$.
\end{lmm}

		\begin{proof}[Proof of~\cref{cor:main}]
		 We will prove all the three parts of~\cref{cor:main} simultaneously by leveraging~\cref{lem:quadratic}. As in the proof of~\cref{thm:main}, it suffices to work with $V_n({\bf X})$ and under the stronger assumption $d_{\square}(W_{Q_n},W)\to 0$. By Lemma \ref{lem:quadratic} part (i), the random variables $V_n({\bf X})=T_{W_{Q_n},\phi}(\widetilde{\mathcal{L}}_n)$ and $T_{W,\phi}(\widetilde{\mathcal{L}}_n)$ are exponentially equivalent, so it suffices to show the LDP for $T_{W,\phi}(\widetilde{\mathcal{L}}_n)$. But this follows from the contraction principle (\cite[Theorem 4.2.1]{DZ}) along with  \cref{lmm:bivarldp}, as the map $T_{W,\phi}(.)$ is continuous by part (ii) of Lemma \ref{lem:quadratic}.
		\end{proof}

		The final prerequisite for the proofs of this section is the following lemma which will be very useful in proving that $I_{\theta}(\cdot)$ (see~\eqref{eq:maingrf}) is a good rate function.
		
		\begin{lmm}\label{lem:contain}
			Suppose $\lVert W\rVert_{q \Delta}<\infty$ for some $q>1$, and assume~\eqref{eq:tailp} holds with $p\ge v$, such that $\frac{1}{p}+\frac{1}{q}\leq 1$. Then for $\theta,\alpha\in \R$ there exists $K>0$ such that
			\begin{equation}\label{eq:astar}
				\mathcal{A}^*_{\alpha}:=
				\{\nu: D(\nu|\rho)-\theta T_{W,\phi}(\nu)\leq \alpha\}\subseteq \{\nu:D(\nu|\rho)\le K\}=\mathcal{A}_{K}.
			\end{equation}
		\end{lmm}
		
		\begin{proof}[Proof of~\cref{lem:gibbs}]
			
			(i) Using Theorem \ref{thm:main}, it follows that $U_n({\bf X})$ satisfies a LDP under $\mu^{\otimes n}$ with the good rate function $I_0(.)$. Since 
			\begin{align*}
				Z_n(\theta)=\frac{1}{n}\log \E_{\mu^{\otimes n}} e^{n\theta U_n({\bf X})},
			\end{align*}
			on invoking Varadhan's lemma (see~\cite[Theorem 4.3.1]{DZ}), assuming we can verify its conditions, we get
			\begin{align*}
				Z_n(\theta)\to \sup_{t\in \R:\ I_0(t)<\infty} \{\theta t-I_0(t)\}=\sup_{\nu\in \widetilde{M}:\  D(\nu|\rho)<\infty} \{\theta T_{W,\phi}(\nu)-D(\nu|\rho)\}<\infty,
			\end{align*}
			where the equality follows from the definition of $I_0$.
			It thus remains to show that Varadhan's Lemma is applicable, for which we need to show (using~\cite[Lemma 4.3.4]{DZ}) that for every $\theta>0$,
			\begin{align}\label{eq:varadhan}
				\limsup_{n\to\infty}\frac{1}{n}\log \E_{\mu^{\otimes n}} e^{n \theta U_n({\bf X})}<\infty.
			\end{align}
			
			To this effect, use~\cref{lem:Tgraphon0} and \eqref{eq:tailp} to get
			{\color{black}\begin{align}\label{eq:ubasebd}
				|U_n({\bf X})|
					&\le \|W_{Q_n}\|_{q\Delta}^{|E(H)|} \Big(\E_{\tilde{\mathfrak{L}}_n }|\phi(B_1,\ldots,B_v)|^p\Big)^{\frac{1}{p}}
     \nonumber \\ &\le \|W_{Q_n}\|_{q\Delta}^{|E(H)|}\Big(\frac{1}{n^v}\sum_{(i_1,\ldots ,i_v)\in [n]^v}\prod_{a=1}^v \psi^p(X_{i_a})\Big)^{\frac{1}{p}}\nonumber \\
					&= \|W_{Q_n}\|_{q\Delta}^{|E(H)|}\Big(\frac{1}{n}\sum_{i=1}^n \psi^p(X_i)\Big)^{\frac{v}{p}}.
				\end{align}} 
				By assumption \eqref{eq:q} the first term in the last line above is bounded in $n$, and so \eqref{eq:varadhan} follows if we can show that for every $\theta>0$, 
				\begin{align}\label{eq:varadhan2}
					\limsup_{n\to\infty}\frac{1}{n}\log \E_{\mu^{\otimes n}} e^{n \theta \Big(\frac{1}{n}\sum_{i=1}^n \psi^p(X_i)\Big)^{\frac{v}{p}}}<\infty.
				\end{align}
				To this effect, for any $t>0$ and $\lambda>1$ we have
				\begin{align*}
					\P\left(\Big(\frac{1}{n}\sum_{i=1}^n\psi^p(X_i)\Big)^{\frac{v}{p}}>t\right)=\P\left(\sum_{i=1}^n\psi^p(X_i)> n t^\frac{p}{v}\right)
					\le e^{-n\lambda t^{\frac{p}{v}}}\Big( \E e^{\lambda \psi^p(X_1)}\Big)^n,
				\end{align*}
				where the RHS is finite by \eqref{eq:tailp}. Since $p\ge v$ and $\E e^{\lambda \psi^p(X_1)}<\infty$ (from \eqref{eq:tailp}), choosing $\lambda>\theta$ along with the last display gives \eqref{eq:varadhan2}, and hence concludes the proof of part (i).
				\\

				
				(ii) Let $\widetilde{T}$ be a bounded continuous function on $\widetilde{\mathcal{M}}$. Then a repeat of the arguments of Theorem \ref{thm:main} gives that $$\theta U_n({\bf X})+\widetilde{T}(\widetilde{\mathcal{L}}_n)=\theta T_{W,\phi}(\widetilde{\mathcal{L}}_n)+\widetilde{T}(\widetilde{\mathcal{L}}_n)$$ satisfies a large deviation principle under $\mu^{\otimes n}$ with the good rate function
				$$\widetilde{I}_\theta(t):=\inf_{\nu\in \widetilde{\mathcal{M}}:\ D(\nu|\rho)<\infty,\ \theta T_{W,\phi}(\nu)+\widetilde{T}(\nu)=t}D(\nu|\rho).$$
				Consequently, using Varadhan's Lemma (as in part (i) above) we get
				\begin{align*}
					\frac{1}{n}\log  \E_{\mu^{\otimes n}} e^{n \theta U_n({\bf X})+n\widetilde{T}(\widetilde{\mathcal{L}}_n)}&\to \sup_{t\in \R:\ \widetilde{I}_{\theta}(t)<\infty} \{ t-\widetilde{I}_{\theta}(t)\}\\ &=\sup_{\nu\in \widetilde{M}:\ D(\nu|\rho)<\infty} \{\theta  T_{W,\phi}(\nu)+\widetilde{T}(\nu)-D(\nu|\rho)\},
				\end{align*}
				which in turn gives
				\begin{align*}
					&\;\;\;\;\frac{1}{n}\log  \E_{\R_{n,\theta}} e^{n\widetilde{T}(\widetilde{\mathcal{L}}_n)}\\&=\frac{1}{n}\log  \E_{\mu^{\otimes n}} e^{n \theta U_n({\bf X})+n\widetilde{T}(\widetilde{\mathcal{L}}_n)}-\frac{1}{n}\log  \E_{\mu^{\otimes n}} e^{n \theta U_n({\bf X})}\\
					&\to \sup_{\nu\in \widetilde{M}:\ D(\nu|\rho)<\infty} \{\theta  T_{W,\phi}(\nu)+\widetilde{T}(\nu)-D(\nu|\rho)\}-\sup_{\nu\in \widetilde{M}:\ D(\nu|\rho)<\infty} \{\theta  T_{W,\phi}(\nu)-D(\nu|\rho)\}\\ &=\sup_{\nu\in \widetilde{M}:\ D(\nu|\rho)<\infty} \{\widetilde{T}(\nu)-I_{\theta}(\nu)\},
				\end{align*}
				where $I_{\theta}$ is defined by \eqref{eq:maingrf}. Now, by~\cite[Theorem 4.4.13]{DZ}, the conclusion will follow if we can show that $I_{\theta}(\cdot)$ is a good rate function.
				
				To prove that $I_{\theta}(\cdot)$ is a good rate function, we will use~\cref{lem:contain}. 
				Towards this direction, given any $\alpha\in\R$, choose $K=K(\alpha,\theta)$ from~\cref{lem:contain} such that~\eqref{eq:astar} holds. Since $D(\cdot|\mu)$ is a good rate function, the set $\mathcal{A}_K$ is compact. On the set $ \mathcal{A}_{K}$, the function $T_{W,\phi}(\cdot)$ can be uniformly approximated by bounded continuous functions (see~\eqref{eq:later}), and is hence continuous. Therefore, $D(\cdot|\rho)-\theta T_{W,\phi}(\cdot)$ is a lower semicontinuous function on the set $\mathcal{A}_K$. In turn, this implies that the set $\mathcal{A}^*_{\alpha}$ is a closed subset of the compact set $\mathcal{A}_K$, and hence compact. Consequently,
				$I_{\theta}(\cdot)$ is a good rate function.
				\\

				For the last part, since the set $\{\nu\in \widetilde{\mathcal{M}}:d_l(\nu,F_\theta)\ge \delta\}$ is closed for any $\delta>0$, we have 
				\begin{align*}
					&\;\;\;\;\;\;\;\;\;\limsup_{n\to\infty}\frac{1}{n}\log{\R_{n,\theta}(d_l(\widetilde{\mathcal{L}}_n,F_\theta)\ge \delta)}\\ &\le -\inf_{\nu\in \widetilde{\mathcal{M}}:d_l(\nu,F_\theta)\ge \delta} \{D(\nu|\rho)-\theta T_{W,\phi}(\nu)\}+ \inf_{\nu\in\widetilde{\mathcal{M}}}\{D(\nu|\rho)-T_{W,\phi}(\nu)\}<0,
				\end{align*}
				where the last inequality follows from the fact that $D(\cdot|\rho)-\theta T_{W,\phi}(\cdot)$ is a lower semicontinuous function, as verified above. The desired conclusion of part (ii) follows.
			\end{proof}
			
			\begin{proof}[Proof of~\cref{lem:con_not}]
				
				(i) Since $U_n({\bf X})$ satisfies a LDP under the measure $\mu^{\otimes n}$ with the good rate function $I_0$ (see~\cref{thm:main}), it follows from \cite[Theorem III.17]{Hollander} that $U_n({\bf X})$ satisfies a LDP under $\R_{n,\theta}$ with the good rate function
				$$J_\theta(t):=\{I_0(t)-\theta t\}-\inf_{z\in \R:\ I_0(z)<\infty}\{I_0(z)-\theta z\}.$$
				We now claim the following:
			\textcolor{black}{	\begin{itemize}
				    \item the unique minimizer of $J_\theta$ is $t_\theta:=T_{W,\phi}(F_\theta)$.
				    \item
				    There exists $K_0>0$ free of $n$ such that for all $K\ge K_0$ we have
			\begin{align}\label{eq:item2}\R_{n,\theta}(|U_n({\bf X})|>K)\le 2 e^{-nK/2}.
				    \end{align}
				\end{itemize}
				Given the first claim, it follows that $U_n({\bf X})\stackrel{P}{\to}t_\theta$ under $\R_{n,\theta}$. Also, by the second claim we have $U_n({\bf X})$ is uniformly integrable, and so
		$Z_n'(\theta)=\E_{\R_{n,\theta}} U_n({\bf X})\to t_\theta.$
					On the other hand, since $Z(\cdot)$ is differentiable at $\theta$ and $Z_n(\cdot)$ is convex, we have $Z_n'(\theta)\to Z'(\theta)$. Consequently we get $Z'(\theta)=t_\theta$, as desired. 
				\\
				To complete the argument, it suffices to verify the two claims above, which we carry out below.
				\begin{itemize}
				    \item  Suppose $t_\theta$ is not the unique global minimizer of $J_\theta$. Then there exists $t_\theta'$ such that $t_\theta\ne t_\theta'$ and $$J_\theta(t_\theta')\le J_\theta(t_\theta)\Rightarrow I_0(t_\theta')-\theta t_\theta'\le I_0(t_\theta)-\theta t_\theta.$$ Since $D(\cdot|\rho)$ is a good rate function, there exists $\nu_\theta'$ such that $$D(\nu_\theta'|\rho)=I_0(t_\theta')\text{ and }T_{W,\phi}(\nu_\theta')=t_\theta'.$$ Consequently, for any $\nu\in F_\theta$ we have
				\[\theta T_{W,\phi}(\nu_\theta')-D(\nu_\theta'|\rho)=\theta t_\theta'-I_0(t_\theta')\ge \theta t_\theta-I_0(t_\theta)\ge \theta T_{W,\phi}(\nu)-D(\nu|\rho),\]
				But since $\nu$ is a global minimizer, it follows that so is $\nu_\theta'$, which is a contradiction, as $T_{W,\phi}(\nu_\theta')=t_\theta'\ne T_{W,\phi}(\nu)$.
				\item
				For any $K>0$ using Markov's inequality gives
				\begin{align*}
				    \R_{n,\theta}(U_n({\bf X})>K)\le e^{-nK}\E_{\R_{n,\theta}}e^{nU_n({\bf X})}=e^{-n(K-Z_n(\theta+1)+Z_n(\theta))}.
				\end{align*}
				Using a similar argument for the lower tail, setting
		$$C:=\sup_{n\ge 1}\max\left\{|Z_n(\theta+1)-Z_n(\theta)|, |Z_n(\theta)-Z_n(\theta-1)|\right\}<\infty $$
		we have 
		$$\R_{n,\theta}(|U_n({\bf X})|>K)\le 2 e^{-n(K-C)}\le e^{-\frac{nK}{2}},$$
		where the last inequality holds for all $K\ge 2C=:K_0$.
				\end{itemize}
				}

				(ii) Fix $\gamma\in F_\theta$. Then $T_{W,\phi}(\gamma)=Z'(\theta)=t$ Let $\gamma'\in \widetilde{\mathcal{M}}$ be arbitrary, such that $T_{W,\phi}(\gamma')=t$. Since $\gamma\in F_\theta$, we have
				\[ \theta T_{W,\phi}(\gamma)-D(\gamma|\rho)\ge \theta T_{W,\phi}(\gamma')-D(\gamma'|\rho).\]
				Also since $T_{W,\phi}(\gamma)=t=T_{W,\phi}(\gamma')$, we get $D(\gamma|\rho)\le D(\gamma'|\rho)$. Thus $$\gamma\in \arg\inf_{\nu\in \widetilde{\mathcal{M}}:T_{W,\phi}(\nu)=t}D(\nu|\rho).$$
				\par	For the other direction, fix $\gamma\in \arg\inf_{\nu\in \widetilde{\mathcal{M}}:\ D(\nu|\rho)<\infty,\ T_{W,\phi}(\nu)=t}D(\nu|\rho).$ Then we have $T_{W,\phi}(\gamma)=t$. Let $\gamma'\in F_\theta$ which implies $T_{W,\phi}(\gamma')=Z'(\theta)=t$ by part (i) above. Then we have $D(\gamma|\rho)\le D(\gamma'|\rho)$, and so
				$$\theta T_{W,\phi}(\gamma)-D(\gamma|\rho)\ge \theta T_{W,\phi}(\gamma')-D(\gamma'|\mu).$$
				But then $\gamma$ is an optimizer of $\theta T_{W,\phi}(\nu)-D(\nu|\rho)$, and so $\gamma\in F_\theta$. This completes the proof of the theorem.

			\end{proof}

			\subsection{Proofs from Sections~\ref{sec:multlin}~and~\ref{sec:tri}}
			
			\begin{proof}[Proof of~\cref{thm:multilinear}]
			We apply Theorem \ref{thm:main} with $\phi(x_1,\ldots,x_v)=\prod_{a=1}^v x_a$ to get that $N_1(H,Q_n,\bf{X})$ satisfies a LDP with good rate function 	
			\begin{equation}\label{eq:multilinear1}
				\inf_{\nu \in \widetilde{\mathcal{M}}:\ D(\nu|\rho)<\infty,\ T_{W,\phi}(\nu)=t} D(\nu|\rho).
			\end{equation}
			We now prove this rate function is equivalent to $I_1(t)$ defined in \eqref{eq:i1}. 
			For any $\nu\in \widetilde{\mathcal{M}}$ such that $D(\nu|\rho)<\infty$, invoking~\cref{lem:basicres} part (i) with $\eta(x)=|x|$, we have $\E_\nu(|V|)<\infty$, and so the function $h_\nu(u):=\E_\nu(V|U=u)$ is well-defined and finite a.s.~Further, invoking~\cref{lem:basicres} part (i) with $\eta(x)=|x|^p$ along with Jensen's inequality we have
			$$\int_{[0,1]}|h_{\nu}(u)|^pdu=\E_{\nu}|\E_\nu(V|U)|^p|\le \E_{\nu}|V|^p<\infty,$$
			as $\E_\mu|X_1|^p<\infty$. Consequently $h_\nu\in \mathcal{L}$. Finally, $T_{W,\phi}$ is well-defined and finite invoking ~\cref{lem:basicres} part (ii), and using Fubini's theorem gives
			\begin{equation}\label{eq:eqty1}
			T_{W,\phi}(\nu)=\int_{[0,1]^v}\prod_{(a,b)\in E(H)} W(u_a,u_b)  \prod_{a=1}^v h_\nu(u_a) du_a=G_{1,W}(h_\nu),
			\end{equation}
			where $G_{1,W}(\cdot)$ is defined as in~\cref{def:g1}. 
			This gives
			\begin{equation}\label{eq:mtof1}
				\inf_{\nu \in \widetilde{\mathcal{M}}:\ D(\nu|\rho)<\infty,\ T_{W,\phi}(\nu)=t} D(\nu|\rho)=\inf_{f\in \mathcal{L}:\ G_{1,W}(f)=t} \inf_{\nu \in \widetilde{\mathcal{M}}: D(\nu|\rho)<\infty,\ h_{\nu}=f} D(\nu|\rho),
			\end{equation}
			To compute the RHS above, fixing $f\in \mathcal{L}$, for any $\nu\in \widetilde{\mathcal{M}}$ we have
			\begin{equation}\label{eq:eqty2}
			D(\nu|\rho)=\int_{[0,1]} D(\nu(V|U=u)|\mu) du.
			\end{equation}
			Fixing $u\in [0,1]$, we have $\E_\nu(V|U=u)=h_\nu(u)=f(u)$. We now consider two cases, depending on whether $f(u)\in \mathcal{N}$ or not, where $\mathcal{N}$ is as in~\cref{def:tilt}.
			
			\begin{itemize}
				\item $f(u)\in \mathcal{N}$.
				
				In this case, among the class of all probability measures $\nu'$ on $\mathcal{X}$ such that $\int_\mathcal{X} x\nu'(dx)=f(u)$, the measure which uniquely maximizes
				$D(\nu'|\mu)$ is an exponential tilt of $\mu$, given by $$\frac{d\nu'_{\rm opt}}{d\mu}(v)=e^{\theta  v-\alpha(\theta)},$$
				where $\theta$ is the unique solution of $\alpha'(\theta)=f(u)$ (cf.~\cite[Section 12.3]{Cover2006}). 
				Equivalently $\theta=\beta(f(u))$, where $\beta$ is the inverse of $\alpha'(\cdot)$. This gives
				$$D(\nu'_{\rm opt}|\mu)= \gamma(\beta(f(u))).$$

				\item $f(u)\in \bar{\mathcal{N}}/\mathcal{N}$.
				
				Without loss of generality assume $f(u)=a:=\sup\{\mathcal{N}\}<\infty$. 
				Since $\mathcal{N}\subset (-\infty,a)$, it follows that $\mu$ is supported on $(-\infty,a]$, and the only measure $\nu'$ supported on $(-\infty,a]$ which satisfies $\int_{\mathcal{X}} x \nu'(dx)=a$ is $\nu'_{\rm opt}=\delta_a$. Consequently, we have
				$$D(\nu'_{\rm opt}|\mu)= D(\delta_a|\mu)= \gamma(\infty)=\gamma(\beta(f(u)),$$
				where we use the definition of $\gamma$ as an extended real valued function in Definition \ref{def:tilt}.

			\end{itemize}

			Combining the two cases above, we have
			\begin{equation}\label{eq:mtof2}
				I_1(t)=\inf_{f\in \mathcal{L}:\ \int_{[0,1]} \gamma(\beta(f(u))) du<\infty,\ G_{1,W}(f)=t} \int_{[0,1]} \gamma(\beta(f(u))) du,
			\end{equation}
			as desired. 
			
			\end{proof}

			\begin{proof}[Proof of~\cref{thm:multilinear_gibbs}]
			(i) The conclusion is an immediate consequence of~\cref{lem:gibbs} part (i), followed by similar calculations as in part (i) to replace the optimization over the space of measures $\widetilde{\mathcal{M}}$ to an optimization over the space of functions $\mathcal{L}$. We sketch this part of the argument for ease of exposition. Given $\nu\in\widetilde{\mathcal{M}}$ such that $T_{W,\phi}(\nu)$ and $D(\nu|\rho)$ are both finite, we have:
			\begin{align}\label{eq:repeat1}
			    \theta T_{W,\phi}(\nu)-D(\nu|\rho)=\theta G_{1,W}(h_{\nu})-\int_{[0,1]} D(\nu(V|U=u)|\mu)\,du,
			\end{align}
			where the above equality uses~\eqref{eq:eqty1}~and~\eqref{eq:eqty2}. From~\eqref{eq:repeat1}, by~\cref{lem:gibbs}, part (i), we further get:
			\begin{equation}\label{eq:repeat2}
			\lim_{n\to\infty} Z_n^{(1)}(\theta)=\sup_{\substack{f\in\mathcal{L},\\ G_{1,W}(f)<\infty}} \left\{\theta G_{1,W}(f)-\inf_{\substack{\nu\in\widetilde{\mathcal{M}},\ h_{\nu}=f\\ D(\nu|\rho)<\infty,\ }} \int_{[0,1]} D(\nu(V|U=u)|\mu)\,du.\right\}
			\end{equation}
			From part (i), the inner infimum in~\eqref{eq:repeat2} simplifies as:
			$$\inf_{\nu\in\widetilde{\mathcal{M}},\ D(\nu|\rho)<\infty,\ h_{\nu}=f} \int_{[0,1]} D(\nu(V|U=u)|\mu)\,du=\int_{[0,1]}\gamma(\beta(f(u)))\,du.$$
			This completes the proof.

			\begin{remark}\label{rem:connection}
			Recall the definitions of $F_{\theta}$ and $F_{\theta}^{(1)}$ from~\cref{lem:gibbs}, part (ii), and~\cref{thm:multilinear_gibbs}, part (ii), respectively. As in~\cref{thm:multilinear}, assume that $\phi(x_1,\ldots ,x_v)=\prod_{i=1}^v x_i$. Two  important conclusions from the above calculation are:
			$$F_{\theta}=\Xi_1(F_\theta^{(1)}),\quad F^{(1)}_{\theta}=\{\E_{\nu}[V|U=u]:\ \nu\in F_{\theta}\}\subseteq \mathcal{L}.$$
			\end{remark}
			
			(ii) The conclusion in~\eqref{eq:multweaklim}, i.e.,  $d_l(\mathfrak{L}_n,\Xi_1(F^{(1)}_\theta))\overset{P}{\longrightarrow}0$ follows directly from~\cref{lem:gibbs}, part (ii), and~\cref{rem:connection}. 	\newline 
			
			(iii) To prove~\eqref{eq:weaaksta}, we first assume that $g(\cdot):[0,1]\to\R$ is a bounded Lipschitz function. Define 
			\begin{equation}\label{eq:deftrun}
				h_M(y):=\begin{cases} -M & \mbox{if}\ y\leq -M \\ y & \mbox{if} -M\leq y \leq M \\ M & \mbox{if}\ y>M,\end{cases}
			\end{equation}
			for some $M>0$, and note that the function $\eta_M:[0,1]\times \R\mapsto \R$ defined by
			$\eta_M(x,y):= g(x) h_M(y)$ is bounded and Lipschitz (with Lipschitz constant depending on $M$). Also we have
			\[ \E_{\tilde{\mathfrak{L}}_n}\eta_M=\int_{[0,1]} g(u)h_M(\omega_n(u))du. \]
			Thus invoking~\eqref{eq:multweaklim} gives
			\begin{equation*}
				\inf_{f\in F^{(1)}_{\theta}}\bigg|\int h_M(\omega_n(u))g(u)\,du - \int h_M(f(u))g(u)\,du\bigg|\overset{P}{\longrightarrow}0,\end{equation*}
			for every fixed $M$. Also note that
			\begin{align*}
			   \bigg|\int h_M(\omega_n(u))g(u)\,du - \int\omega_n(u)g(u)\,du\bigg|\le& \|g\|_\infty\frac{1}{n}\sum_{i=1}^n|X_i| 1\{|X_i|\ge M\},\\
			    \bigg|\int h_M(f(u))g(u)\,du - \int f(u)g(u)\,du\bigg|\le &\|g\|_\infty \int |f(u)|1\{|f(u)|\ge M\}\, du.
			\end{align*}
			Given the above two displays, to show \eqref{eq:weaaksta} for $g$ Lipschitz, it suffices to show that
			for any $\delta>0$,
			\begin{align*}
		\limsup_{M\to\infty}\limsup_{n\to\infty} \R_{n,\theta}^{(1)}\Big(\sum_{i=1}^n |X_i| 1\{|X_i|>M\}>n\delta\Big)&=0,\\
				\limsup_{M\to\infty}\sup_{f\in F_\theta^{(1)}} \int_{[0,1]}|f(u)|1\{|f(u)|>M\}\,du&=0.
			\end{align*}
			These follow, if we can show 
			\begin{align}
			\label{eq:lip11}\sum_{i=1}^n \E_{\R_{n,\theta}^{(1)}}|X_i|^p&=O(n),\\
				\label{eq:lip12}\sup_{f\in F_\theta^{(1)}} \|f\|_p&<\infty.
			\end{align}
			Deferring the proofs of \eqref{eq:lip11} and \eqref{eq:lip12}, we now prove \eqref{eq:weaaksta} for a general $g\in L^{p^*}[0,1]$. To this effect,
			let $\{g_\ell\}_{\ell\ge 1}$ be a sequence of bounded Lipschitz continuous functions (with bounds and Lipschitz constants depending on $\ell$), such that $g_\ell \stackrel{L^{p^*}[0,1]}{\to}g$. Then for any $f\in F_\theta^{(1)}$, triangle inequality and H\"{o}lder's inequality gives
			\begin{align*}
				\bigg|\int ( \omega_n -f)g\bigg|\le & \bigg|\int \omega_n(g-g_\ell)\bigg|+\bigg|\int (\omega_n-f)g_\ell\bigg|+\bigg|\int f(g_\ell-g)\bigg|\\
				\le &\|\omega_n\|_{p}\|g-g_\ell\|_{p^*}+\bigg|\int (\omega_n-f)g_\ell\bigg|+\|f\|_p \|g_\ell-g\|_{p^*},
			\end{align*}
			which on taking an infimum over $f\in F_\theta^{(1)}$ bounds $\inf_{f\in F_\theta^{(1)}}\bigg|\int ( \omega_n -f)g\bigg|$ by
			\begin{align}\label{eq:lip1}
		 \|\omega_n\|_{p}\|g-g_\ell\|_{p^*}+\inf_{f\in F_\theta^{(1)}}\bigg|\int (\omega_n-f)g_\ell\bigg|+\Big(\sup_{f\in F_\theta^{(1)}}\|f\|_p\Big) \|g_\ell-g\|_{p^*}.
			\end{align}
			On letting $n\to \infty$, the second term in \eqref{eq:lip1} converges in probability to $0$, using~\eqref{eq:weaaksta} for Lipschitz functions. The first term converges in probability to $0$ as $n\to\infty$ followed by $\ell\to\infty$, since invoking \eqref{eq:lip11} gives
			\begin{align*}
		\|\omega_n\|_p^p=\frac{1}{n}\sum_{i=1}^n|X_i|^p=O_P(1)
			\end{align*}
			under $\R_{n,\theta}$. Finally, the third term in the RHS converges to $0$ as $\ell\to\infty$, on invoking \eqref{eq:lip12}.
			\\
			
			To complete the proof, it thus remains to verify \eqref{eq:lip11} and \eqref{eq:lip12}. These are done below:
			
			\begin{itemize}
				\item{Proof of \eqref{eq:lip11}}

			To begin, fixing $C:=\sup_{n\ge1}|Z_n(\theta)|<\infty$ (see~\cref{lem:gibbs}, part (i)), for $M,K>0$ (to be chosen later) we have:
				\begin{align*}
		&{\R_{n,\theta}^{(1)}}\left(\frac{1}{n}\sum_{i=1}^n |X_i|^{p}\geq M\right) \\\leq& {\R_{n,\theta}^{(1)}}(|U_n({\bf X})|\geq K)+\exp(nK |\theta|+nC)\P_{\mu^{\otimes n}}\left(\frac{1}{n}\sum_{i=1}^n |X_i|^{p}\geq M\right)\nonumber \\
		\leq& 2\exp(-nK/2)+\exp(nK|\theta|+nC-nM)\Big(\E_{\mu}\exp(|X_1|^{p})\Big)^n\nonumber\\
		\leq& 2\exp(-nK/2)+\exp(2nK|\theta|-nM)\nonumber\\
		\notag\le & 2e^{-nK/2}+e^{-nM/2},
				\end{align*}
				where the second inequality uses \eqref{eq:item2} for $K\ge K_0$,
				the third inequality needs $K|\theta|\ge C+\log \E_\mu [\exp(|X_1|^p)])$, and the fourth inequality needs $M\ge 4K|\theta|$. Thus with $$K\ge \max\Big[K_0,\frac{C+\log \E_\mu [\exp(|X_1|^p)])}{|\theta|}\Big]$$ we get
		$${\R_{n,\theta}^{(1)}}\left(\frac{1}{n}\sum_{i=1}^n |X_i|^{p}\ge \frac{K}{4|\theta|}\right)\le 2e^{-nK/2}+e^{-2nK|\theta|}.$$		
			From this	\eqref{eq:lip11} follows on integrating over $K$.
				
				
				\item{Proof of \eqref{eq:lip12}}
				
				If the conclusion does not hold, there exists a sequence of functions $\{f_k\}_{k\ge 1}$ in $F_\theta^{(1)}$ such that $\lim_{k\to\infty}\|f_k\|_p\to \infty$. Set $\nu_k:=\Xi_1(f_k)$, and note that $f_k(u)=\E_{\nu_k}[V|U=u]$. From~\cref{rem:connection}, it follows that $\nu_k\in F_{\theta}$, where $F_{\theta}$ is the set of optimizers defined as in~\cref{lem:gibbs}, part (ii). Thus there exists $\alpha\in \R$ such that $\sup_{k\ge 1}I_\theta(\nu_k)\le \alpha$. Using~\cref{lem:contain} we get the existence of $K<\infty$ such that $\sup_{k\ge 1}D(\nu_k|\rho)\le K$. Invoking~\cref{lem:basicres} with $\eta(x)=|x|$, we then get
				$$\sup_{k\ge 1}\E_{\nu_k^{(2)}}[|V|^p]\le \sup_{\nu:D(\nu|\rho)\le K}\E_{\nu^{(2)}}[|V|^p]\le C(K,\mu)<\infty.$$
				Using Jensen's inequality along with the above display, we then get
				$$\sup_{k\geq 1}\lVert f_k\rVert_{p}=\sup_{k\geq 1} \lVert \E_{\nu_k}[V|U]\lVert_p\leq \sup_{\nu\in\mathcal{A}_{K}} \left(\E_{\ntw}|V|^p\right)^{\frac{1}{p}}\leq C(K,\mu)^{1/p}<\infty,$$
				which is a contradiction to the assumption that $\|f_k\|_p\to \infty$. 
				This completes the proof.
			\end{itemize}

						\end{proof}
			
			\begin{proof}[Proof of~\cref{thm:monochromatic}]
			We apply Theorem \ref{thm:main} with $\phi(x_1,\ldots,x_v)=1\{x_1=x_2=\ldots=x_v\}$ to get that $N_2(H,Q_n, \bf{X})$ satisfies a LDP with good rate function 	
			\begin{equation}
				\inf_{\nu \in \widetilde{\mathcal{M}}:\ D(\nu|\rho)<\infty,\ T_{W,\phi}(\nu)=t} D(\nu|\mu).
			\end{equation}
			We now prove this rate function is equivalent to $I_2(t)$ defined in \eqref{eq:i2}. To this end, for any $\nu\in \widetilde{\mathcal{M}}$, setting $h_{\nu,r}(u):=\P_\nu(V=r|U=u)$ for $r\in [c]$ we see that ${\bf h}_\nu:=(h_{\nu,r})_{r\in [c]}\in \mathcal{F}_c$, and  $$T_{W,\phi}(\nu)=\sum_{r=1}^c\int_{[0,1]^v}\prod_{(a,b)\in E(H)} W(u_a,u_b)  \prod_{a=1}^v h_{\nu,r}(u_a) du_a=G_{2,W}({\bf h}_\nu).$$
			
			\noindent Conversely, given any ${\bf f}\in \mathcal{F}_c$, there is a unique  measure $\nu\in \widetilde{\mathcal{M}}$ such that ${\bf f}(u)={\bf h}_\nu(u)$. 
			Indeed, such a $\nu$ can be obtained as the law of $(U,V)$, where $U\sim \mathrm{U}[0,1]$, and given $U$, let $V$ be a random variable taking values in $[c]=\{1,2,\ldots,c\}$ with probabilities $\{f_1(U),\ldots,f_c(U)\}$ (which adds upto $1$ by definition of $\mathcal{F}_c$.
			Finally, for any $\nu\in \widetilde{\mathcal{M}}$ we have
			$$D(\nu|\mu)=\int_{[0,1]} D(\nu(V|U=u)|\mu) du=\int_{[0,1]}\sum_{r=1}^c h_{\nu,r}(u)\log \frac{h_{\nu,r}(u)}{\mu_r}du$$ This gives
			\begin{eqnarray*}
				I_2(t)=	&\inf_{\nu \in \widetilde{\mathcal{M}}:\ D(\nu|\rho)<\infty,\ T_{W,\phi}(\nu)=t} D(\nu|\mu)\\
				=&\inf_{{\bf f}\in \mathcal{F}_c:\ G_{2,W}({\bf f})=t} \Big\{\int_{[0,1]}\sum_{r=1}^c h_{\nu,r}(u)\log\frac{ h_{\nu,r}(u)}{\mu_r}du\Big\},
			\end{eqnarray*}
			as desired.
			
			\end{proof} 
			
			\begin{proof}[Proof of~\cref{thm:monochromatic_gibbs}]
			
			The conclusions in~\cref{thm:monochromatic_gibbs} use the same arguments as the equivalent conclusions in~\cref{thm:multilinear_gibbs}.
			\newline

			\end{proof}

			\section{Proofs of Auxiliary Results}\label{sec:auxlem}
			
			We start with the following definition.
			\begin{defn}
				For any $W\in \mathcal{W}$,  
				and $\mathbf{f}=\{f_a, 1\le a\le v\}$ measurable functions with $f_a:\mathcal{X}\mapsto \R $, define
				\[
				t(H,W,{\bf f})= \int_{[0,1]^{v}}\prod_{(a,b)\in E(H)} W(x_a,x_b) \prod_{a=1}^v f_a(x_a) dx_a,
				\]
				if the above integral exists as a finite real number.
				
			\end{defn}
			We now state the following proposition, which will be used in proving the lemmas of~\cref{sec:main}. The proof of this proposition is provided alongside.
			\begin{prop}\label{lem:gen_holder}
				Let ${\bf f}:\mathcal{X}^v\mapsto [-1,1]$, and $q>1$.
				
				(i)	If $W$ be a graphon such that $\|W\|_{q\Delta}<\infty$, 
				then $t(H,W,{\bf f})$ is well-defined, and $|t(H,W,{\bf f})|\le \|W\|^{|E(H)|}_{q\Delta}.$
				\\
				
				(ii) Let $\{W_n\}_{n \ge 1}, W$ be graphons such that \eqref{eq:q} holds, and $\|W_n-W\|_{\square} \rightarrow 0$.	 Then $$\lim\limits_{n \rightarrow \infty}\sup_{{\bf f}:\mathcal{X}^v\mapsto [-1,1]^v} |t(H,W_n,{\bf f})-t(H,W,{\bf f})| =0.$$
				
			\end{prop}
			
			\begin{proof}[Proof of~\cref{lem:gen_holder}]
				\emph{Part (i).} To begin, use H\"older's inequality to get
				\begin{align}\label{eq:hold}
					|t(H,W,{\bf f})|\le \Big(\int_{[0,1]^v} \prod_{(a,b)\in E(H)} |W(x_a,x_b)|^q \prod_{a=1}^v dx_a\Big)^{\frac{1}{q}}\le  \|W\|^{|E(H)|}_{q\Delta},
				\end{align}
				where the second inequality uses \cite[Proposition 2.19]{bc_lpi}.

				\emph{Part (ii).} 
				The fact that $t(H,W_n,{\bf f})$ and $t(H,W,{\bf f})$ are both well-defined follows from (i). The convergence  follows by replacing $t(F,W)$ with $t(H,W,\mathbf{f})$ in~\cite[Thm 2.20]{bc_lpi}. To be specific,  in~\cite[(8.3)]{bc_lpi}, one needs to replace $a(x_{i_t})b(x_{i_t})$ by  $a(x_{i_t})b(x_{j_t})f_{i_t}(x_{i_t})f_{j_t}(x_{j_t})$, and note that  $\|a(x_{i_t})f_{i_t}(x_{i_t})\|_{\infty} \le 1$, $\|a(x_{j_t})f_{j_t}(x_{j_t})\|_{\infty} \le 1$. The rest of the argument goes through, completing the proof.

			\end{proof}

			\begin{proof}[Proof of Lemma \ref{lem:Tgraphon0}]
				
				Set $q^*=\frac{q}{q-1}$. H\"older's inequality implies
			\small{	\begin{align*}
					T_{|W|,|\phi|}(\nu)\le& \Big(\int_{[0,1]^v} \prod_{(a,b)\in E(H)}\Big|W(x_a, x_b)\Big|^{q} \prod_{a=1}^v dx_a\Big)^{\frac{1}{q}} \Big(\E_{(\ntw)^{\otimes v}} |\phi(B_1,\ldots,B_v)|^{q^*}\Big)^{\frac{1}{q^*}}.
				\end{align*}}
				As $\frac{1}{p}+\frac{1}{q}\leq 1$, we have $p\geq q^*$. By a further application of H\"{o}lder's inequality, we can then replace $q^*$ in the above display by $p$. With this observation, using~\eqref{eq:hold}, the first term in the RHS above can be bounded by $ \|W\|^{|E(H)|}_{q\Delta}$. Lemma \ref{lem:Tgraphon0} follows. 
				
				%

		\end{proof}

	\begin{proof}[Proof of Lemma \ref{lem:phim}]
		
		Setting $\widetilde{\phi}_M:=\phi-\phi_M$ and then using~\cref{lem:Tgraphon0} with
		$\phi$ replaced by $\widetilde{\phi}_M$ and a general $W(\cdot,\cdot)$, we get
		\begin{align}\label{eq:Holder1}
			|T_{W,\widetilde{\phi}_M}(\nu)|
				\le &\|W\|^{|E(H)|}_{q\Delta} \Big(\E_{\nu} |\widetilde{\phi}_M(B_1,\ldots,B_v)|^p\Big)^{\frac{1}{p}}.
			\end{align}
			To bound the second term in the RHS above, {\color{black} using \eqref{eq:tailp}}, we get
			\begin{align*}
				|\widetilde{\phi}_M(B_1,\ldots,B_v)|=&1\{|\phi(B_1,\ldots,B_v)|>M\} |\phi(B_1,\ldots,B_v)| \\
				\le &1\{\max_{a\in [v]}\psi(B_a)>M^{1/v}\}\prod_{b=1}^v \psi(B_b) ,
			\end{align*}
			which gives
			\begin{align}
			\notag\E_{\nu} |\widetilde{\phi}_M(B_1,\ldots,B_v)|^p 
				\le& \sum_{a=1}^v \E_{\nu} \left(\prod_{b=1}^v \psi^p(B_b) \mathbbm{1}_{\psi^p(B_a) > M^{1/v}}\right) \\
				= &v \Big(\E_{\nu} \psi^p\Big)^{v-1} \E_{\nu}\psi^p\mathbbm{1}_{\psi^p > M^{1/v}} \label{eq:phi_m_bound_ii}
			\end{align} 
			Combining \eqref{eq:Holder1} and \eqref{eq:phi_m_bound_ii} we get
			\begin{align}\label{eq:hold+phim}
				|T_{W,\widetilde{\phi}_M}(\nu)|\le v\|W\|^{|E(H)|}_{q\Delta}  \Big(\E_{\nu} \psi^p\Big)^{v-1} \E_{\nu}\psi^p\mathbbm{1}_{\psi^p > M^{1/v}}.
			\end{align}
			
			(i) Observe that $V_n({\bf X})-V_{n,M}({\bf X})=T_{W_{Q_n},\widetilde{\phi}_M}({\bf X})$.  To show part (i), using $W=W_{Q_n}$ and $\nu=\widetilde{\mathcal{L}}_n$ in \eqref{eq:hold+phim} we get the bound
			\begin{align*}
				|T_{W_{Q_n},\widetilde{\phi}_M}(\widetilde{\mathcal{L}}_n)|\le v\|W_{Q_n}\|^{|E(H)|}_{q\Delta}\Big(\frac{1}{n}\sum_{i=1}^n\psi(X_i)^p\Big)^{v-1}\Big(\frac{1}{n}\sum_{i=1}^n\psi^p(X_i) 1\{\psi^p(X_i)>M^{1/v}\}\Big).
			\end{align*}
			
			Since the first term in the RHS of \eqref{eq:phi5} is bounded in $n$ by \eqref{eq:q}, to verify part (i) it suffices to show the following:
			\begin{eqnarray}
				\label{eq:phi5}\limsup_{K\to\infty}\limsup_{n\to\infty}\frac{1}{n}\log \P\Big(\sum_{i=1}^n\psi^p(X_i)>nK\Big)=&-\infty,\\
				\label{eq:phiii6}
				\limsup_{M\to\infty}\limsup_{n\to\infty}\frac{1}{n}\log \P\Big(\sum_{i=1}^n\psi^p(X_i)1\{\psi(X_i)>M^{1/v}\}>n\delta\Big)=&-\infty.
			\end{eqnarray}
			To this effect,
						{\color{black} Markov's inequality yields that for any $K>0$, 
      \begin{align*}
          \P(\sum_{i=1}^n\psi^p(X_i)>nK)= \P \Big(\exp\Big(\sum_{i=1}^n\psi^p(X_i)\Big)> e^{nK}\Big) \le  e^{-nK} \E\Big( e^{ \psi(X_1)^p}\Big)^n,
      \end{align*}
      which, on taking logarithm and dividing by $n$ yields 
      $$\frac{1}{n}\log \mathbb{P}\Big(\sum_{i=1}^n\psi^p(X_i)>nK\Big) \le - K+ \log \Big( \E e^{ \psi(X_1)^p}\Big),$$}
							the RHS of which converges to $-\infty$ as $K \rightarrow \infty$ using \eqref{eq:tailp}. This verifies \eqref{eq:phi5}. A similar calculation gives that for any $\delta,M,\lambda>0$,
							\begin{align*}
								\frac{1}{n}\log \mathbb{P}\left(\sum_{i=1}^n\psi^p(X_i)1\{\psi(X_i)>M^{1/v}\} \ge n\delta\right) \le - \lambda \delta+ \log \Big( \E e^{ \lambda\psi(X_1)^p\mathbbm{1}_{\psi^p(X_i) > M^{1/v}}}\Big),
							\end{align*}
							the RHS of which converges to $-\lambda \delta$ as $M \rightarrow \infty$ using \eqref{eq:tailp} and dominated convergence theorem, for every fixed $\lambda>0$. On letting $\lambda\to\infty$ this goes to $-\infty$, 
							which verifies \eqref{eq:phiii6}, and hence completes the proof of part (i).

							In a similar manner, starting from~\eqref{eq:ubasebd} with $\phi$ replaced by $\widetilde{\phi}_M$, the conclusion with $U_n({\bf X})$ and $U_{n,M}({\bf X})$ follows. We omit the details for brevity.

							(ii) Using \eqref{eq:hold+phim}, to show the desired conclusion it suffices to show the following:
							$$\sup_{\nu\in \mathcal{A}_K} \int \psi^pd\nu <\infty, \quad\lim_{M\to\infty}\sup_{\nu \in \mathcal{A}_K}\int \psi^p 1\{\psi>M^{1/v}\} d\nu=0.$$
							Both of these follow on invoking~\cref{lem:basicres} with $\eta=\psi$ and $\eta=\psi 1\{\psi>M^{1/v}\}$ respectively, on noting that
							$(\phi,\mu)$ satisfies \eqref{eq:tailp}.

						\end{proof}

						\begin{proof}[Proof of Lemma \ref{lem:Tcont}]
							
							(i) Let $\{\nu_n\}_{n\ge 1},\nu\in  \widetilde{\mathcal{M}}$ be such that $\nu_n$ converges to $\nu$ in the weak topology. 
						    \textcolor{black}{As $\lVert W^*\rVert_1<\infty$, therefore there exists continuous functions $\{W_{\ell}\}_{\ell\ge 1}$ on $[0,1]^v$ such that $\lVert W_{\ell}-W^*\rVert_1\to 0$. Given a function $\widetilde{W}\in L_1([0,1]^v)$ and a bounded continuous function $\phi:\mathcal{X}^v\mapsto \R$, define a functional $\widetilde{T}_{\widetilde{W},\phi}$ on $\widetilde{\mathcal{M}}$ by setting
          $$\widetilde{T}_{\widetilde{W},\phi}(\nu):=\E\left[\phi(B_1,\cdots,B_v) \widetilde{W}(A_1,\cdots,A_v)\right],$$
          where the expectation is over $\{(A_a,B_a)\}_{1\le a\le v}\stackrel{i.i.d.}{\sim}\nu$. With this definition, we have 
          $T_{W,\phi}(\nu)=\widetilde{T}_{W^*,\phi}(\nu)$.
          An application of the triangle inequality then gives:
							\begin{align*}
								&|T_{W,\phi}(\nu_n)-T_{W,\phi}(\nu)|\\
        =&|\widetilde{T}_{W^*,\phi}(\nu_n)-\widetilde{T}_{W^*,\phi}(\nu)|\\
        \le &|\widetilde{T}_{W^*,\phi}(\nu_n)-\widetilde{T}_{W_\ell,\phi}(\nu_n)|+|\widetilde{T}_{W_\ell,\phi}(\nu_n)-\widetilde{T}_{W_\ell,\phi}(\nu)|+|\widetilde{T}_{W_\ell,\phi}(\nu)-\widetilde{T}_{W,\phi}(\nu)|\\
								\le& |\widetilde{T}_{W_\ell,\phi}(\nu_n)-\widetilde{T}_{W_\ell,\phi}(\nu)|+ 2 \lVert \phi\rVert_{\infty}\lVert W_{\ell}-W^*\rVert_1.
							\end{align*}}
							The first term converges to $0$ as $n\to\infty$ using the definition of weak convergence and an application of the dominated convergence theorem, as both $W_{\ell}$ and $\phi$ are bounded continuous functions. The second term converges to $0$ as $\ell\to\infty$. This completes the proof of part (i).
							\\ 
							
							(ii)  Setting $\tilde{\phi}_\ell:=\phi-\phi^{(\ell)}$, by~\cref{lem:Tgraphon0} we have
							$$|T_{W,\phi^{(\ell)}}(\nu)-T_{W,\phi}(\nu)|=|T_{W,\tilde{\phi}_\ell}|\le \|W\|_{q\Delta}^{|E(H)|}\Big(\E_{\nu^{\otimes v}}|\tilde{\phi}_\ell(B_1,\ldots,B_v)|^p\Big)^{\frac{1}{p}} .$$
							Since $\|W\|_{q\Delta}<\infty$, it suffices to show that
							$$\lim_{\ell\to\infty}\sup_{\nu\in \mathcal{A}_K}\E_{\nu^{\otimes v}}|\tilde{\phi}_\ell(B_1,\ldots,B_v)|^p=0.$$
							But this follows on noting that $$\|\tilde{\phi}_\ell\|_\infty\le 2M, \quad \tilde{\phi}_\ell\stackrel{L^p(\mu)}{\longrightarrow}0,$$ and invoking~\cref{lem:basicres}, part (i) with $\eta=\tilde{\phi}_\ell$.

						\end{proof}
						
						\begin{proof}[Proof of Lemma \ref{lem:Tgraphon}]
							 Without loss of generality, assume $\|\phi\|_{\infty}\le 1$. Invoking~\cref{lem:Tgraphon0} with $\phi\equiv 1$ (as in~\eqref{eq:l1w}), we get that there exists $C>0$ such that \begin{align}\label{eq:pivoteq}
								\max(\max_{n\ge 1}\|W_n^*\|_1,\|W^*\|_1)\le C,
								\end{align} where $$W_n^*(x_1,\ldots,x_v):=\prod_{(a,b)\in E(H)}W_{Q_n}(x_a,x_b),\quad W^*(x_1,\ldots,x_v)=\prod_{(a,b)\in E(H)}W(x_a,x_b),$$ 
							
							Fixing $\varepsilon>0$, let $\mathcal{K}\subseteq\mathcal{X}$ be a compact set such that $\mu(\mathcal{K}^c)\le \varepsilon$. Fixing $\delta>0$, let $\widetilde{\phi}:\mathcal{K}^v\mapsto \R$ be a measurable function, such that
							$$\widetilde{\phi}({\bf x})=\sum_{i=1}^{k} \alpha_i \mathbbm{1}_{{\bf x} \in \prod_{a=1}^{v} R_{ia}}, \text{ and }\quad \sup_{{\bf x} \in \mathcal{K}^v}|\phi({\bf x})-\tilde{\phi}({\bf x})|\le \frac{\delta}{4C},$$
	{\color{black} where $\{R_{ia}\}_{i \in [k], a\in[v]}$ are measurable subsets of $\mathcal{X}$.} 						
       The existence of such a function follows from the Stone-Weierstrass Theorem, as $\mathcal{K}^v$ is compact, and the class of functions of the form $$\Big\{{\bf x}\mapsto \sum_{i=1}^k\alpha_i 1_{{\bf x}\in \prod_{a=1}^v R_{ia}}, k\in \N, (R_{ia})_{i\in [k], a\in [v]}\subset \mathcal{K}^{vk}\Big\}$$
							is closed under pointwise multiplication, contains constant functions, and separates points in $\mathcal{K}^v$. We note that both $\mathcal{K}$ and $\widetilde{\phi}(\cdot)$  depend on $\vep$, but we omit the dependence for notational simplicity. 
							
							For $i\in [k], a\in [v]$ and $\nu \in \mathcal{A}_K$, define a function $f_{ia,\nu}:[0,1]\mapsto [0,1]$ by setting $f_{ia,\nu}(x):= \mathbb{P}_\nu(B \in R_{ia}|A=x)$, where $(A,B)\sim \nu$. Then
			\begin{align*}T_{W_n,\widetilde{\phi}}(\nu)=&\sum_{i=1}^{k}\alpha_i\E\Bigg[\left(\prod_{(a,b)\in E(H)}W_n(A_a,A_b)\right)\left(\prod_{a=1}^{v} f_{ia}(A_a)\right)\Bigg]\\
			=&\sum_{i=1}^k \alpha_i t(H,W_n,{\mathbf f}_i),
			\end{align*}
							where ${\mathbf f}_i:=(f_{ia},a\in [v])$ is a map from $\mathcal{X}^v$ to $[-1,1]^v$. The above display gives
							\begin{align}\label{eqn:2bound}
						\notag	&	\sup_{\nu\in \widetilde{\mathcal{M}}}|T_{W_n,\widetilde{\phi}}(\nu)-T_{W,\widetilde{\phi}}(\nu)|\\
								\le&\Big( \sum_{i=1}^k |\alpha_i| \Big)\sup_{{\bf f}:\mathcal{X}^v\mapsto [-1,1]^v} |t(H,W_n,{\bf f})-t(H,W,{\bf f}) | \le \frac{\delta}{4},
							\end{align}
							for all $n$ large, by part (ii) of \cref{lem:gen_holder}. Also, since $\|\phi-\tilde{\phi}\|_\infty<\varepsilon$ on $\mathcal{K}^v$, extending $\tilde{\phi}$ to $\mathcal{X}^v$ by setting $\tilde{\phi}=0$ outside $\mathcal{K}^v$, we have 
							\begin{align*}
								\left \lvert T_{W_{Q_n},\phi}(\tilde{\mathfrak{L}}_n)-T_{W_{Q_n},\tilde{\phi}}(\tilde{\mathfrak{L}}_n)\right \rvert\le &\|W_n^*\|_1 \E_{\widetilde{\mathfrak{L}_n}}|\phi'|\\
								\le &C\Big(\frac{\delta}{4C}+\frac{1}{n^v}\sum_{i_1,\ldots,i_v=1}^n 1\{X_{i_r}\in \mathcal{K}^c\text{ for some }r\in [v]\}\Big)\\
								\le &\frac{\delta}{4}+\frac{Cv}{n}\sum_{i=1}^n 1\{X_i\in \mathcal{K}^c\}.
							\end{align*}
							Using the above display along with \eqref{eqn:2bound}~and~\eqref{eq:pivoteq}, we get
							$$\left\lvert T_{W_{Q_n},\phi}(\tilde{\mathfrak{L}}_n)-T_{W,\phi}(\tilde{\mathfrak{L}}_n)\right\rvert\le \frac{3\delta}{4}+\frac{2Cv}{n}\sum_{i=1}^n1\{X_i\in \mathcal{K}^c\},$$
							which gives
							\begin{align*}
								\P\Big(|T_{W_{Q_n},\phi}(\tilde{\mathfrak{L}}_n)-T_{W,\phi}(\tilde{\mathfrak{L}}_n)|>\delta\Big)\le &\P\left(\sum_{i=1}^n1\{X_i\in \mathcal{K}^c\}>\frac{\delta}{8Cv}\right).
							\end{align*}
							Since $\P(X_1\in \mathcal{K}^c)\le \varepsilon$, using standard Binomial concentration bounds give
							$$\limsup_{\varepsilon\to 0}\limsup_{n\to\infty}\frac{1}{n}\log \P\left(\sum_{i=1}^n1\{X_i\in \mathcal{K}^c\}>\frac{\delta}{8Cv}\right)=-\infty,
							$$ from which the desired conclusion is immediate.
							
						\end{proof}
						
						\begin{proof}[Proof of~\cref{lem:basicres}]
							
							\noindent \emph{Part (i).} Fix $\nu \in \mathcal{A}_K$ (as in~\eqref{eq:ak}) and define $h= \frac{d \ntw}{d \mu}$ (where $\ntw$ is as in~\eqref{eq:secmar}). Then with $\rho=\mathrm{Unif}[0,1]\otimes \mu$ as before, we have
			\begin{equation}\label{eq:firstineq}
			    D(\nu^{(2)}|\mu)\leq D(\nu|\rho)\le K.
			\end{equation}
			Then, for any $M_1,M_2>1$, setting $G:= \int |\eta|^pd\mu$ we have
			\begin{align}\label{eq:final_term}
				\notag	\int |\eta |^p d\ntw= & \int |\eta|^p h d\mu \\ 
				\notag	=& \int |\eta|^p h \mathbbm{1}_{h \le M_1} d\mu+ \int |\eta|^p h \mathbbm{1}_{h > M_1}\mathbbm{1}_{|\eta|\le M_2} d\mu +\int |\eta|^p h \mathbbm{1}_{h > M_1}\mathbbm{1}_{|\eta|> M_2} d\mu\\
				\notag	 \le& M_1 G + \frac{M^p_2}{\log M_1} \int h \log h  \mathbbm{1}_{h > M_1} d\mu+\int |\eta|^p h \mathbbm{1}_{h > M_1}\mathbbm{1}_{|\eta|> M_2} d\mu\\
				\le& M_1G + \frac{M^p_2}{\log M_1} (K+4)+\int |\eta|^p h \mathbbm{1}_{h > M_1}\mathbbm{1}_{|\eta|> M_2} d\mu.
			\end{align}
			Here the last inequality uses the bound
			\begin{align*}
				\int h\log h 1\{h>M_1\} d\mu\le &D(\nu^{(2)}|\mu)-\int h\log h 1\{h\le 1\} d\mu\\
				\le &K-\sum_{i=1}^\infty \int h\log h 1\Big\{e^{-i}< h\le e^{-(i-1)}\Big\} d\mu\\
				\le &K+\sum_{i=1}^\infty i e^{-(i-1)}\mu\Big(e^{-i}< h\le e^{-(i-1)}\Big)\\
				\le &K+\sum_{i=1}^\infty ie^{-(i-1)}\le K+4.
			\end{align*}
			The second inequality in the above display uses~\eqref{eq:firstineq}. 
			To bound the third term in the RHS of \eqref{eq:final_term}, first note that for any $y>0$ we have
			$$\sup_{x\in \R}\{xy-e^x\}=y\log y-y.$$
			Thus, for any $x\in \R$ and $y>0$ we have the inequality
			$$xy\le e^x+y\log y-y\le e^x+y\log y.$$
	
	Using this, for any $\lambda>1$ we have
			\begin{align}\label{eq:b3}
				\notag\int |\eta|^p h \mathbbm{1}_{h > M_1}\mathbbm{1}_{|\eta|> M_2} d \mu  \le& \int \left(e^{\lambda |\eta|^p}+ \frac{h}{\lambda} \log h \right)\mathbbm{1}_{h > M_1}\mathbbm{1}_{|\eta|> M_2} d\mu\\
				\le &\int e^{\lambda |\eta|^p}\mathbbm{1}_{|\eta|> M_2}d\mu+\frac{K+4}{\lambda}.
			\end{align}
			Now, first pick $\lambda=\frac{K+4}{G}$ such that the second term in the RHS of \eqref{eq:b3} equals $G$. Since $\int e^{\alpha |\eta|^p}d\mu \le \mathfrak{G}(\alpha)$ for all $\alpha>0$, we have
			\begin{align}\label{eq:b4}
				\int e^{\lambda |\eta|^p}\mathbbm{1}_{|\eta|> M_2}d\mu \le e^{-\lambda M^p_2}\int e^{(\lambda+1) |\eta|^p}d\mu \le e^{-\lambda M^p_2}\mathfrak{G}(\lambda+1).
			\end{align}
			Choose $M_2:=\left(-\frac{1}{\lambda}\ln\frac{G}{\mathfrak{G}(\lambda+1)}\right)^{1/p}$, ensuring that the RHS of \eqref{eq:b4} equals $G$.
			Finally choose $M_1=\exp\Big(\frac{M_2^p(K+4)}{G}\Big)$, such that the second term in the RHS of \eqref{eq:final_term} equals $G$. With all these choices, combining \eqref{eq:final_term}, \eqref{eq:b3} and \eqref{eq:b4} gives
			\begin{align*}
				& \sup_{\nu \in \mathcal{A}_K} \int |\eta|^pd\ntw \le (M_1+3)G<\infty.
			\end{align*}
			From this the desired conclusion follows with $L=M_1+3$. Here $M_1$ depends on $K, \mathfrak{G}$ and $G$, but this is not a problem, as $G\le \mathfrak{G}(1)$. 
			
							\vspace{0.1in}
							
							\noindent \emph{Part (ii).} For this part, an application of the triangle inequality, coupled with the observation $$\{\nu:D(\nu|\rho)<\infty\}=\cup_{K\geq 0}\mathcal{A}_K$$ implies that it suffices to show~\eqref{eq:welldef} holds for $T_{|W|,|\phi|}(\cdot)$ instead of $T_{W,\phi}(\cdot)$. Towards this direction, using~\cref{lem:Tgraphon0} and~\eqref{eq:tailp}, we get:
							$$\sup_{\nu\in\mathcal{A}_K}T_{|W|,|\phi|}\leq \lVert W\rVert_{q \Delta}^{|E(H)|}\left(\E_{\ntw}\psi(X_1)^p\right)^{\frac{v}{p}}.$$
							A direct application of~\cref{lem:basicres}, part (i) with $\eta=\psi$, completes the proof.
						\end{proof}

						\begin{proof}[Proof of~\cref{lmm:un_vn_same}] 
						 We begin with the following simple observation:
						\begin{align}\label{eq:unsame2}
						\limsup\limits_{n\to\infty} \,\, \lVert W_{Q_n}\rVert_{q\Delta}<\infty \quad \implies \quad \limsup\limits_{n\to\infty} \,\, \lVert W_{1\vee Q_n}\rVert_{q\Delta}<\infty.
						\end{align}
						Fix $\tilde{q}\in (1,q)$, and let $\delta:=\frac{q}{\tilde{q}}-1>0$, and  define 		$Q^*_{n,\delta}(i,j):=1\vee |Q_n(i,j)|^{1+\delta}$ for all $i,j$. 
						Then using~\cref{lem:Tgraphon0} with $\phi\equiv 1$ gives 
	\begin{align}\label{eq:unsame1}	
	& \frac{1}{n^v} \sum_{(i_1,\ldots,i_v)\in [n]^v}\prod_{(a,b)\in E(H)}|Q^*_{n,\delta}(i_a,i_b)|\nonumber\\
	\le& \|W_{Q^*_{n,\delta}}\|^{|E(H)|}_{\tilde{q}\Delta}=\lVert W_{1\vee Q_{n}}\rVert_{\tilde{q}(1+\delta)\Delta}^{(1+\delta)|E(H)|}=\lVert W_{1\vee Q_{n}}\rVert_{q\Delta}^{(1+\delta)|E(H)|},
					\end{align}
					where we use the fact that  $\tilde{q}(1+\delta)=q$.
						{\color{black} For $M>1$, define $Q_{n,M}(i,j):=Q_n(i,j)\mathbf{1}(|Q_n(i,j)|\leq M)$ and $\Delta_{n,M}(i,j)= Q_{n}(i,j)-Q_{n,M}(i,j)$for all $i,j$. In the sequel, we will use the following simple inequalities:
					\begin{align}\label{eq:unsame3}
					|\Delta_{n,M}(i,j)|\leq \frac{1}{M^{\delta}}Q^*_{n,\delta}(i,j),\quad \max \{|Q_n(i,j)|,|Q_{n,M}(i,j)|\}\leq Q_{n,\delta}^*(i,j).
					\end{align}
					Then by writing $Q_n(i,j)=Q_{n,M}(i,j)+\Delta_{n,M}(i,j)$, we get
		\begin{align}\label{eq:unsame4}					   &\;\;\;\;\frac{1}{n^v}\Bigg\lvert \sum_{(i_1,\ldots,i_v)\in [n]^v} \phi(x_{i_1},\ldots ,x_{i_v}) \Bigg(\prod_{(a,b)\in E(H)}Q_n(i_a,i_b) - \prod_{(a,b)\in E(H)}Q_{n,M}(i_a,i_b) \Bigg)\Bigg\rvert\nonumber \\
  & \le \frac{L}{n^v} \sum_{(i_1,\ldots,i_v)\in [n]^v} \Bigg\lvert \Bigg(\prod_{(a,b)\in E(H)}\Big[Q_{n,M}(i_a,i_b) +\Delta_{n,M}(i_a,i_b)\Big]- \prod_{(a,b)\in E(H)}Q_{n,M}(i_a,i_b) \Bigg) \Bigg\rvert \nonumber\\
  &= \frac{L}{n^v} \sum_{(i_1,\ldots,i_v)\in [n]^v} \Bigg\lvert \sum_{S \subseteq E(H), S \neq \emptyset} \prod_{(a,b) \in S} \Delta_{n,M}(i_a,i_b) \prod_{(a,b) \in E(H) \setminus S^c} Q_{n,M}(i_a,i_b)  \Bigg\rvert \nonumber\\
  &\leq \frac{2^{|E(H)|}}{M^{\delta}}\frac{L}{n^v} \ \sum_{(i_1,\ldots,i_v)\in [n]^v}\prod_{(a,b)\in E(H)}Q^*_{n,\delta}(i_a,i_b)\nonumber \\ &\leq \frac{2^{|E(H)|}L}{M^{\delta}}\lVert W_{Q_{n,\delta}^*}\rVert_{\tilde{q}\Delta}^{|E(H)|}.
						\end{align}
					The inequality in the fourth line in the above display uses~\eqref{eq:unsame3} and the last inequality uses \eqref{eq:unsame1}.} Combining \eqref{eq:unsame2}, \eqref{eq:unsame1} and \eqref{eq:unsame4} we have 
		\begin{align}\label{eq:unsame50}		 &\lim\limits_{M\to\infty}\limsup\limits_{n\to\infty}\sup_{(x_1,\ldots,x_n)\in \mathcal{X}^n}\frac{1}{n^v}\Bigg\lvert \sum_{(i_1,\ldots,i_v)\in [n]^v }\phi(x_{i_1},\ldots ,x_{i_v})\Bigg(\prod_{(a,b)\in E(H)}Q_n(i_a,i_b) - \nonumber \\ &\qquad\qquad\qquad \prod_{(a,b)\in E(H)}Q_{n,M}(i_a,i_b) \Bigg)\Bigg\rvert=0.
					\end{align}
					A similar argument with $[n]^v$ replaced by
				 $\mathcal{S}_{n,v}$ 
				gives	\begin{align}\label{eq:unsame5}		 &\lim\limits_{M\to\infty}\limsup\limits_{n\to\infty}\sup_{(x_1,\ldots,x_n)\in \mathcal{X}^n}\frac{1}{n^v}\Bigg\lvert \sum_{(i_1,\ldots,i_v)\in \mathcal{S}(n,v) }\phi(x_{i_1},\ldots ,x_{i_v})\Bigg(\prod_{(a,b)\in E(H)}Q_n(i_a,i_b) - \nonumber \\ &\qquad\qquad\qquad \prod_{(a,b)\in E(H)}Q_{n,M}(i_a,i_b) \Bigg)\Bigg\rvert=0.
					\end{align}
					Finally, note that there exists positive real constant $C_v$ such that $\left|[n]^v\setminus \mathcal{S}(n,v)\right|\leq C_v n^{v-1}$. Consequently, for any fixed $M>1$, we have:
					\begin{align*}					    & \frac{1}{n^v}\Bigg\lvert \sum_{(i_1,\ldots,i_v)\in [n]^v }\phi(x_{i_1},\ldots ,x_{i_v})\prod_{(a,b)\in E(H)}Q_{n,M}(i_a,i_b) \\ &\quad \quad- \sum_{(i_1,\ldots,i_v)\in \mathcal{S}(n,v)}\phi(x_{i_1},\ldots ,x_{i_v})\prod_{(a,b)\in E(H)}Q_{n,M}(i_a,i_b)\Bigg\rvert\\ &\leq\frac{L}{n^v} \sum_{(i_1,\ldots ,i_v)\in [n]^v\setminus \mathcal{S}(n,v)} \prod_{(a,b)\in E(H)} |Q_{n,M}(i_a,i_b)| \leq  \frac{L M^{|E(H)|}C_v}{n},
					\end{align*}
					which gives
						\begin{align*}				    &\limsup_{n\to\infty}\sup_{(x_1,\ldots,x_n)\in \mathcal{X}^v} \frac{1}{n^v}\Bigg\lvert \sum_{(i_1,\ldots,i_v)\in [n]^v }\phi(x_{i_1},\ldots ,x_{i_v})\prod_{(a,b)\in E(H)}Q_{n,M}(i_a,i_b) \\ &\quad \quad- \sum_{(i_1,\ldots,i_v)\in \mathcal{S}(n,v)}\phi(x_{i_1},\ldots ,x_{i_v})\prod_{(a,b)\in E(H)}Q_{n,M}(i_a,i_b)\Bigg\rvert=0.
						\end{align*}
					Combining the above display along with~\eqref{eq:unsame50}~and~\eqref{eq:unsame5} completes the proof.

						
						

						\end{proof}
						
						\begin{proof}[Proof of Lemma \ref{lem:quadratic}]
							We will assume without loss of generality throughout the proof that $\lVert\phi\rVert_{\infty}\leq 1$. Also let $U\sim\mathrm{U}[0,1]$.
							\\
							
							(i) \emph{(Under~\eqref{eq:pp}).} Note that $T_{|W|,|\phi|}(\nu)\leq \lVert W\rVert_1<\infty$ which implies $T_{W,\phi}(\cdot)$ is well-defined. Fixing $\varepsilon>0$, invoking Stone-Weierstrass Theorem as in the proof of Lemma \ref{lem:Tgraphon}, we obtain the existence of a function $\widetilde{\phi}$ of the form
							\begin{equation}\label{eq:quadclaim}\widetilde{\phi}(x,y)=\sum_{i=1}^k\alpha_i 1\{x\in R_{i1}, y\in R_{i2}\},\end{equation}
							where $R_{i1}, R_{i2}$ are measurable subsets of $\mathcal{X}$, such that $\|\phi-\widetilde{\phi}\|_\infty<\varepsilon$. 
							This gives
							$$|T_{W_n,\phi}(\nu)-T_{W_n,\widetilde{\phi}}(\nu)|\le \varepsilon \|W_n\|_1,\quad |T_{W,\phi}(\nu)-T_{W,\widetilde{\phi}}(\nu)|<\varepsilon \|W\|_1,$$
							which on using triangle inequality gives
							\begin{align}\label{eq:triangle}
								|T_{W_n,\phi}(\nu)-T_{W,\phi}(\nu)|\le |T_{W_n,\widetilde{\phi}}(\nu)-T_{W,\widetilde{\phi}}(\nu)|+\varepsilon\Big( \|W\|_1+ \|W_n\|_1\Big).
							\end{align}
							Finally, using the form of $\widetilde{\phi}$, setting $f_{ia,\nu}(u):=\P_\nu(V\in R_{ia}|U=u)$, for $i\in [k]$ and $a=1,2$, we get
							\begin{align}\label{eq:form1}
								\notag|T_{W_n,\widetilde{\phi}}(\nu)-T_{W,\widetilde{\phi}}(\nu)|\le &\sum_{i=1}^k |\alpha_i|\Big| \int_{[0,1]^2} (W_n(x,y)-W(x,y))f_{i1,\nu}(x) f_{i2}(y)dxdy\Big|\\
								\le &d_\square(W_n,W)\Big(\sum_{i=1}^k |\alpha_i|\Big).
							\end{align}
							Combining \eqref{eq:triangle} and \eqref{eq:form1} the conclusion of part (i) follows using \eqref{eq:pp}, as $\varepsilon>0$ is arbitrary.
							\\
							
							\emph{(Under~\eqref{eq:pp1}).} Label the vertices of $H=K_{1,v-1}$, such that the central vertex gets label $1$. For any $\widetilde{W}\in\mathcal{W}$ such that $\E\mar_{\widetilde{W}}(U)<\infty$, 
							integrating first over the variables $x_2,\ldots ,x_v$ 
							we have
							\begin{align}\label{eq:ppi1}
							\int_{[0,1]^v}\prod_{i=2}^v |\widetilde{W}(x_1,x_j)|\prod_{j=1}^v \,dx_i\leq \E \mar^{v-1}_{\widetilde{W}}(U).
							\end{align}
							Once again, fixing $\vep>0$, invoking Stone-Weierstrass Theorem, we obtain the existence of $\widetilde{\phi}$ of the form
							\begin{align}\label{eq:ppi2}
							\widetilde{\phi}(x_1,\ldots ,x_v)=\sum_{i=1}^k \alpha_i\prod_{j=1}^v 1\{x_j\in R_{ij}\},
							\end{align}
							where $R_{i1},\ldots ,R_{iv}$ are measurable subsets of $\mathcal{X}$, such that $\lVert \phi-\widetilde{\phi}\rVert_{\infty}\leq \vep$. Using~\eqref{eq:ppi1} with $\widetilde{W}\in \{W_{Q_n}, W\}$, triangle inequality gives
							\begin{small}
							\begin{align}\label{eq:ppi3}			\left|T_{W_{Q_n},\phi}(\nu)-T_{W,\phi}(\nu)\right|\leq \left|T_{W_{Q_n},\widetilde{\phi}}(\nu)-T_{W,\widetilde{\phi}}(\nu)\right|+\vep \E\left( \mar^{v-1}_{W_{Q_n}}(U)+ \mar^{v-1}_{W}(U)\right).
							\end{align}
							\end{small}

As in the proof of~\cref{lem:Tgraphon}, for $i\in [k], a\in [v]$ and $\nu\in \widetilde{\cM}$  define a function $f_{ia,\nu}:[0,1]\mapsto [0,1]$ by setting $f_{ia,\nu}(x):= \mathbb{P}_\nu(B \in R_{ia}|A=x)$, where $(A,B)\sim \nu$. Then we have
			\begin{align*}T_{W_{Q_n},\widetilde{\phi}}(\nu)=&\sum_{i=1}^{k}\alpha_i\E\Bigg[\left(\prod_{(a,b)\in E(H)}W_{Q_n}(A_a,A_b)\right)\left(\prod_{a=1}^{v} f_{ia}(A_a)\right)\Bigg]\\
			=&\sum_{i=1}^k \alpha_i t(H,W_{Q_n},{\mathbf f}_i),
			\end{align*}
		and where ${\bf f}_i:=(f_{ia})_{a\in [b]}$. Consequently, a telescopic calculation gives
							\begin{align}\label{eq:form}
						&\;\;\;\;\left|T_{W_{Q_n},\widetilde{\phi}}(\nu)-T_{W,\widetilde{\phi}}(\nu)\right|\nonumber \\ 
						&\le \sum_{i=1}^k |\alpha_i|\Big| \int_{[0,1]^v} \left(\prod_{a=2}^v W_{Q_n}(x_1,x_a)-\prod_{a=2}^v W(x_1,x_a)\right)\prod_{a=1}^v f_{ia,\nu}(x_a)\, dx_a\Big|\nonumber \\
							\notag	&\le \sum_{i=1}^k |\alpha_i|\sum_{a=2}^{v}\Bigg|\int_{[0,1]^2} \Big(W_{Q_n}(x_1,x_{a})-W(x_1,x_{a})\Big)f_{i1,\nu}(x_1)f_{ia,\nu}(x_{ \ell})\,dx_1\,dx_{a}\nonumber \\ 
						\nonumber	&\quad\quad \int_{[0,1]^{v-2}}\left(\prod_{r=2}^{a-1} W(x_1,x_{r})f_{ir,\nu}(x_{r})\,dx_{r}\right)\left(\prod_{s=a+1}^{v} W_{Q_n}(x_1,x_s)f_{is,\nu}(x_{s})\,dx_{s}\right)\Bigg|\\
								&= \sum_{i=1}^k |\alpha_i|\sum_{a=2}^{v}\Bigg|\int_{[0,1]^2} \left(W_{Q_n}(x_1,x_{a})-W(x_1,x_{a})\right)f_{i1,\nu}(x_1)f_{ia,\nu}(x_{ \ell})\,dx_1\,dx_{a}\nonumber \\ &\quad\quad \left(\prod_{r=2}^{a-1} \mathcal{R}_{ir}(x_1;\nu)\right)\left(\prod_{s=a+1}^{v}\mathcal{R}^{(n)}_{is}(x
								_a;\nu) \right)\Bigg|
							\end{align}
							where $\mathcal{R}^{(n)}_{ia}(x;\nu):=\int_{[0,1]} W_{Q_n}(x,y)f_{ia,\nu}(y)\,dy$ for $i\in [k], a\in [v]$, and $\mathcal{R}_{ia}(x;\nu)$ is defined similarly with $W_{Q_n}$ replaced by $W$. Observe that 
							\begin{align}\label{eq:ppi4}
							|\mathcal{R}_{ia}^{(n)}(x;\nu)|\leq \mar_{W_{Q_n}}(x),\quad \quad |\mathcal{R}_{ia}(x;\nu)|\leq \mar_{W}(x).
							\end{align}
						Intersecting with the set $\max(\mar_{W_{Q_n}}(x_1), \mar_{W}(x_a))> L$ and its complement, using~\eqref{eq:form} and~\eqref{eq:ppi4}, for every fixed $L>0$ we have
							\begin{align*} &\;\;\;\;\left|T_{W_{Q_n},\widetilde{\phi}}(\nu)-T_{W,\widetilde{\phi}}(\nu)\right|\nonumber \\
							&\leq \left(\sum_{i=1}^k |\alpha_i|\right)\bigg[\sum_{a=2}^{v}\E\bigg[\mar_{W_{Q_n}}^{v-a}(U)\mar_{W}^{a-2}(U)\left(\mar_{W_{Q_n}}(U)+\mar_{W}(U)\right)\bigg(1(\mar_{W_{Q_n}}(U)>L)\nonumber \\&\quad\quad+1(\mar_{W}(U)>L)\bigg)\bigg] +L^{v-2}d_{\square}\left(W_{Q_n},W\right)\bigg]. 
							\end{align*}
							Observe that the above bound is free of $\nu$, and consequently by using the assumption $d_{\square}(W_{Q_n},W)\to 0$, we have:
							\begin{align}\label{eq:ppi5}
							    &\;\;\;\;\limsup_{n\to\infty}\sup_{\nu\in\widetilde{\mathcal{M}}} \left|T_{W_{Q_n},\widetilde{\phi}}(\nu)-T_{W,\widetilde{\phi}}(\nu)\right|\nonumber \\&\leq \left(\sum_{i=1}^k |\alpha_i|\right)\bigg[\sum_{a=2}^{v}\limsup_{n\to\infty}\E\bigg[\mar_{W_{Q_n}}^{v-a}(U)\mar_{W}^{a-2}(U)\left(\mar_{W_{Q_n}}(U)+\mar_{W}(U)\right)\nonumber \\&\quad\quad\bigg(1(\mar_{W_{Q_n}}(U)>L)+1(\mar_{W}(U)>L)\bigg)\bigg]\bigg].
							\end{align}
							Now, for fixed $a$, the summation in the right hand side of~\eqref{eq:ppi5} splits canonically into four summands. We will only bound one of them as the other three can be bounded similarly. Define $\tp:=(v-1)/(v-a+1)$, $\tq:=(v-1)/(a-2)$, and note that $\tp^{-1}+\tq^{-1}=1$. H\"{o}lder's inequality gives
							\begin{align*}
							&\;\;\;\; \E\left[\mar_{W_{Q_n}}^{v-\ell+1}(U)\mar_W^{\ell-2}(U) 1(\mar_{W_{Q_n}}(U)>L)\right]\\ &\leq  \left(\E\left[\mar_{W_{Q_n}}^{(v-\ell+1)\tp}(U)1(\mar_{W_{Q_n}}(U)>L)\right]\right)^{\frac{1}{\tp}}\left(\E\left[\mar_W^{(\ell-2)\tq}(U)\right]\right)^{\frac{1}{\tq}}\\ &=\left(\E\left[\mar_W^{v-1}(U)\right]\right)^{\frac{1}{\tq}}\E\left[\mar_{W_{Q_n}}^{v-1}(U)1(\mar_{W_{Q_n}}(U)>L)\right].
							\end{align*}
							The RHS above converges to $0$ on letting $L\to\infty$, using~\eqref{eq:pp1}. Combining the above observation with~\eqref{eq:ppi5}, we get:
							\begin{align}\label{eq:ppi6}
							\limsup_{L\to\infty}\limsup_{n\to\infty}\sup_{\nu\in\widetilde{\mathcal{M}}}\left|T_{W_{Q_n},\widetilde{\phi}}(\nu)-T_{W,\widetilde{\phi}}(\nu)\right|=0.
							\end{align}
							Combining the above display with~\eqref{eq:ppi3} completes the proof, as $\vep>0$ is arbitrary.
							\\
							
							\emph{(Under~\eqref{eq:pp2})} To begin, use~\eqref{eq:pp2} to note the existence of a finite positive constant $C$ such that \begin{align}\label{eq:mathfrak}
							\|\mathfrak{R}_n\|_\infty\le C,\text{ where }\mathfrak{R}(x):=\max\left(|\mar_{W_{Q_n}}(x)|,|\mar_{W}(x)|\right).
							\end{align}
							We now claim that under~\eqref{eq:ppi2}, for any tree $\widetilde{H}$ with  vertex set $[k]$ any $a_0\in [k]$, and any $\{e_{ab}\}_{(a,b)\in E(\widetilde{H})}\in \{0,1\}^{k-1}$, 
							the following holds:
							\begin{equation}\label{eq:pp21}
							    \int_{[0,1]^{k-1}} \prod_{(a,b)\in E(\widetilde{H})}
						\big|W_{Q_n}^{e_{ab}}(x_a,x_b)\big|^{e_{ab}}\big|W(x_a,x_b)\big|^{1-e_{ab}}\prod_{a\neq a_0}\,dx_a\leq C^{k-1}.					
							    \end{equation}
							We prove the claim by induction. If $k=2$, then $\widetilde{H}=K_2$ is an edge, and so the LHS of \eqref{eq:pp21} is one of the following: $$\{\mar_{W_{Q_n}}(x),\quad \mar_{W_{Q_n}}(y),\quad \mar_{W}(x),\quad \mar_{W}(x)\}.$$
							All terms in the above display are bounded by $C$, by~\eqref{eq:mathfrak}, and so \eqref{eq:pp21} holds for $k=2$.
						Assume \eqref{eq:pp21} holds for some value of $k\ge 2$. Proceeding to verify the result for a tree $\widetilde{H}$ with $k+1$ vertices, given any $a_0\in \widetilde{H}$, there exists $\ta_0\neq a_0$ such that $\ta_0$ is a leaf node of $\widetilde{H}$ (this holds since $\widetilde{H}$ is not an edge). Thus there exists a node $\tb_0$ in the vertex set of $\widetilde{H}$ such that $(\ta_0,\tb_0)\in E(\widetilde{H})$ is the only edge in $\widetilde{H}$ connected to $\ta_0$. Therefore, by integrating over $x_{\ta_0}$ first, we get:
							\begin{align*}
							   &\;\;\;\; \int_{[0,1]^{k-1}} \prod_{(a,b)\in E(\widetilde{H})} \big|W_{Q_n}(x_a,x_b)\big|^{e_{ab}}\big|W(x_a,x_b)\big|^{1-e_{ab}}\prod_{j\in [k]/ a_0}\,dx_j\\ &\leq C\left(\int_{[0,1]^{k-2}} \prod_{(a,b)\in E(\widetilde{H}(\ta_0))} \big|W_{Q_n}(x_a,x_b)\big|^{e_{ab}}\big|W(x_a,x_b)\big|^{1-e_{ab}}\prod_{j\in [k]\{a_0,\ta_0\}}\,dx_j\right),			\end{align*}
          where $\tH(\ta_0)$ is the sub-tree of $\tH$ formed be removing the vertex $\ta_0$, and all the edges connected to $\ta_0$.
							   As $\widetilde{H}(\ta_0)$ is a tree with $k-1$ vertices, we can bound the term in within the parentheses by $C^{k-2}$ using the induction hypothesis. This completes the induction step and proves~\eqref{eq:pp21}. It further shows that 
          \begin{align}\label{eq:finite_big}
          \max\Big(\|W_n^*\|_1, \|W\|_1\Big)\le  C^{v-1}. 
         \end{align}
          and so
          $T_{W_{Q_n},\phi}(\nu)$ and $T_{W,\phi}(\nu)$ are finite and well-defined.
							 \\
							 
							Let $\widetilde{\phi}$ be as in~\eqref{eq:ppi2}. By using the triangle inequality with~\eqref{eq:finite_big}, we get
							\begin{equation}\label{eq:pp22}
							\left|T_{W_{Q_n},\phi}(\nu)-T_{W,\phi}(\nu)\right|\leq \left|T_{W_{Q_n},\widetilde{\phi}}(\nu)-T_{W,\widetilde{\phi}}(\nu)\right|+2\vep C^{v-1}.
							\end{equation}
							Consider any arbitrary edge  $(a^*,b^*)\in E(H)$. Then $H\setminus (a^*,b^*)$ splits into two disjoint trees, one containing the vertex $a^*$ and the other containing the vertex $b^*$ (it is possible that one of these components is an isolated vertex). Let $H_{a^*}$ and $H_{b^*}$ denote these two connected components respectively. With $\{f_{ia}(\nu)\}_{i\in [k], a\in [v]}$ as before, by repeating computations similar to~\eqref{eq:form} we get
							\begin{align}\label{eq:pp23}
							    &\;\;\;\;\left|T_{W_{Q_n},\widetilde{\phi}}(\nu)-T_{W,\widetilde{\phi}}(\nu)\right|\nonumber\\ &\leq \sup_{\substack{\small{(a^*,b^*)\in E(H)},\\ \small{\mathbf{e}\in \{0,1\}^{|E(H)/(a^*,b^*)|}}}}v \sum_{i=1}^k |\alpha_i|\Bigg|\int\limits_{[0,1]^2}\Big(W_{Q_n}(x_{a^*},x_{b^*})-W(x_{a^*},x_{b^*})\Big)f_{ia^*,\nu}\big(x_{a^*}\big) f_{ib^*,\nu}(x_{b^*})\nonumber \\&\;\;\; \int_{[0,1]^{| V(H_{a^*})\setminus a^*|}}\prod_{(a,b)\in E(H_{a^*})} W_{Q_n}^{e_{ab}}(x_a,x_b)W^{1-e_{ab}}(x_a,x_b)\prod_{[0,1]^{| V(H_{a^*})\setminus a^*|}} f_{ij,\nu}(x_j)\,dx_j\nonumber \\&\;\;\; \int_{[0,1]^{| V(H_{b^*})\setminus b^*|}}\prod_{(a,b)\in E(H_{b^*})} W_{Q_n}^{e_{ab}}(x_a,x_b)W^{1-e_{ab}}(x_a,x_b)\prod_{[0,1]^{| V(H_{b^*})\setminus b^*|}} f_{ij,\nu}(x_j)\,dx_j\Bigg|.
							\end{align}

							As $H_{a^*}$ and $H_{b^*}$ are  trees with $|V(H_{a^*})|$ and $|V(H_{b^*})|$ vertices, by applying~\eqref{eq:pp21} with $\widetilde{H}=H_{a^*},H_{b^*}$ and $a_0=a^*,b^*$, we get:
							\begin{small}\begin{align*}\bigg| \int\limits_{[0,1]^{| V(H_{a^*})\setminus a^*|}}\prod_{(a,b)\in E(H_{a^*})} W_{Q_n}^{e_{ab}}(x_a,x_b)W^{1-e_{ab}}(x_a,x_b)\prod_{[0,1]^{| V(H_{a^*})\setminus a^*|}} f_{ij,\nu}(x_j)\,dx_j\bigg|\leq C^{|V(H_{a^*})|-1},\\
							\bigg| \int\limits_{[0,1]^{| V(H_{b^*})\setminus b^*|}}\prod_{(a,b)\in E(H_{b^*})} W_{Q_n}^{e_{ab}}(x_a,x_b)W^{1-e_{ab}}(x_a,x_b)\prod_{[0,1]^{| V(H_{b^*})\setminus b^*|}} f_{ij,\nu}(x_j)\,dx_j\bigg|\leq C^{|V(H_{b^*})|-1}.
       \end{align*}\end{small}
							Further the left hand sides of the inequalities above, without the modulus, are only functions of $x_{a^*}$ and $x_{b^*}$ respectively. 
							Using the definition of strong cut in~\cref{def:defirst}, the above display implies:
							$$\sup_{\nu\in\widetilde{\mathcal{M}}}\left|T_{W_{Q_n},\widetilde{\phi}}(\nu)-T_{W,\widetilde{\phi}}(\nu)\right|\leq v C^{v-2} d_{\square}(W_{Q_n},W)\sum_{i=1}^k |\alpha_i|\to 0,$$
							using the assumption that $d_{\square}(W_{Q_n},W)\to 0$. Combining the above display with~\eqref{eq:pp22} completes the proof.
							\\
							
							(ii)  \emph{(Under~\eqref{eq:pp})} This follows directly by combining~\cref{lem:Tcont} part (i) with~\eqref{eq:pp}.
							\\
							
							\emph{(Under~\eqref{eq:pp1})} This follows directly by combining~\cref{lem:Tcont} part (i) with~\eqref{eq:pp1}, on invoking~\eqref{eq:ppi1} with $\widetilde{W}=W$.
							\\
							
							\emph{(Under~\eqref{eq:pp2})}  This follows directly by combining~\cref{lem:Tcont} part (i) with~\eqref{eq:pp2}, on invoking~\eqref{eq:finite_big}.			
							\end{proof}
						
						\begin{proof}[Proof of~\cref{lem:contain}]	
							Let us start by assuming the contrary. This implies that there exists $\alpha$ and a sequence of measures $\{\nu_k\}_{k\geq 1}$ such that $\nu_k\in\mathcal{A}^*_{\alpha}$ but $D(\nu_k|\rho)\to\infty$. Without loss of generality, assume $\theta\geq 0$. Since $D(\nu_k|\rho)-\theta T_{W,\phi}(\nu_k)\le \alpha$, we must have $T_{W,\phi}(\nu_k)\to\infty$. Take any $\theta'>\theta$. Then
						\begin{align*}	\theta'T_{W,\phi}(\nu_k)-D(\nu_k|\rho)&=(\theta T_{W,\phi}(\nu_k)-D(\nu_k|\rho))+(\theta'-\theta)T_{W,\phi}(\nu_k)\\ &\geq -\alpha+(\theta'-\theta)T_{W,\phi}(\nu_k)\to\infty,
						\end{align*}
							which contradicts the conclusion from~\cref{lem:gibbs} part (i). This completes the proof. 
						\end{proof}

						\begin{proof}[Proof of \cref{rem:ex}]

			 \begin{itemize}
			     \item {$d_\square(W_{Q_n},W)=o_P(1).$}

			     To begin, any $S,T\subseteq[n]$ we have
$$\sum_{i\in S, j\in T}\Big(B_{ij}-\frac{1}{2}\Big)\left(\frac{ij}{n^2}\right)^{-\alpha}={\rm subg}\Big(\frac{1}{4}\sum_{i\in S, j\in T}\left(\frac{ij}{n^2}\right)^{-2\alpha}\Big)\le {\rm subg}(C n^{2}), $$
for some finite positive constant $C$ free of $n$.  Using this, standard concentration and union bounds gives
$$\P\left(\max_{S,T\subseteq[n]}\Big(Q_n(i,j)-\widetilde{Q}_n(i,j)\Big)>n^2\delta\right)\le \sum_{S,T\subseteq [n]} 2 e^{-\frac{n^2\delta^2}{C}}=2^{4n+1}e^{-\frac{n^2\delta^2}{C}},$$
 where $\widetilde{Q}_n(i,j)=\frac{1}{2}\Big(\frac{ij}{n^2}\Big)^{-\alpha}.$ 
 This gives
$$d_\square(W_{Q_n}, W_{\widetilde{Q}_n})=\frac{1}{n^2}\max_{S,T\subseteq[n]}\Big|\sum_{i\in S, j\in T}(Q_n(i,j)-\widetilde{Q}_n(i,j))\Big|=o_P(1),$$
and so it suffices to show that $d_\square(\widetilde{W}_{Q_n},W)\to 0$. But this follows on using monotonicity of the function $x\mapsto x^{-\alpha}$ on $[0,1]$ to note that
\begin{align}\label{eq:l1_act}\|W_{\widetilde{Q}_n}-W\|_1=\int_{[0,1]^2} (xy)^{-\alpha} dxdy-\frac{1}{2n^2}\sum_{i,j=1}^n\Big(\frac{ij}{n^2}\Big)^{-\alpha} \rightarrow 0,
\end{align}
which holds for all $\alpha\in (0,1)$.
\\

\item{$\|W_{Q_n}-W\|_1\stackrel{P}{\to}\frac{1}{2(1-\alpha)^2}>0. $}

Using \eqref{eq:l1_act}, it suffices to show that 
$\|\|W_{Q_n}-W_{\widetilde{Q}_n}\|_1\stackrel{P}{\to} \frac{1}{2(1-\alpha)^2}$. But this follows on noting that
\begin{align*}
    \|W_{Q_n}-W_{\widetilde{Q}_n}\|_1&=\frac{1}{n^2}\sum_{i,j=1}^n \left(\frac{ij}{n^2}\right)^{-\alpha} \left\lvert B_{ij}-\frac{1}{2}\right\rvert \\
    &=\frac{1}{2n^2} \sum_{i,j=1}^n \left(\frac{ij}{n^2}\right)^{-\alpha} \rightarrow \frac{1}{2(1-\alpha)^2}.
\end{align*}

\item{$\limsup_{n\to\infty}\|W_{Q_n}\|_{q\Delta}<\infty$ for some $q>1$.}

For any $q\ge 1$  we have
\begin{align*}
    \int |W_{Q_n}|^{q\Delta}\le \frac{1}{n^2}\sum_{i,j=1}^n \Big(\frac{ij}{n^2}\Big)^{-\alpha q\Delta}\to \frac{1}{(1-\alpha q\Delta)^2},
\end{align*}
where the last convergence holds for $q\in \Big[1, \frac{1}{\alpha \Delta}\Big)$.

\item{$\|W_{Q_n}\|_{\infty} \stackrel{P}{\to}\infty$.}

For any $i\in [n]$ such that $B(i,1)=1$ we have
$$W_{Q_n}\Big(\frac{i}{n},\frac{1}{n}\Big)=\Big(\frac{i}{n^2}\Big)^{-\alpha}\ge n^{\alpha}.$$
Consequently, 
$$\P\Big(\|W_{Q_n}\|_\infty<n^{\alpha}\Big)\le \P\Big(B_{i1}=0\text{ for }2\le i\le n)=\frac{1}{2^{n-1}}\to 0,$$
which completes the proof.
\end{itemize}
			\end{proof}

\end{document}